\documentclass[12pt]{amsart}
\usepackage{amsthm}
\usepackage{latexsym,amssymb}
\usepackage{graphicx}
\usepackage{mathrsfs}
\usepackage{amsmath}
\newcommand{\G}{\Gamma}

\newcommand{\inv}[1]{\mbox{${#1}^{-1}$}}
\newcommand{\norm}[1]{\|{#1}\|}
\newcommand{\w}[1]{\widehat{#1}}

\newcommand{\mcF}{\mathcal F}
\newcommand{\mcH}{\mathcal H}

\newcommand{\eps}{\varepsilon}
\newcommand{\mcD}{\mathcal{D}}
\newcommand{\mcV}{\mathcal{V}}

\newcommand{\N}{\mbox{${\mathbb{N}}$}}

\newcommand{\C}{\mbox{${\mathbb  C}$}}

\newcommand{\Om}{\Omega}

\newcommand{\1}{\mbox{${{\bf 1}}$}}
\newcommand{\period}{\;.}

\newcommand{\tr}{\mathrm{tr}}
\newcommand\atopn[2]{\genfrac{}{}{0pt}{}{#1}{#2}}

\date{\today}

\newtheorem{theorem_intro}{Theorem}

\newtheorem{thm}{Theorem}[section]
 \newtheorem*{reftheorem1}{Theorem \reftotheorem}
\newenvironment{reftheorem}[1]
  {\newcommand{\reftotheorem}{\ref{#1}}\begin{reftheorem1}}
  {\end{reftheorem1}}
 \newtheorem{coro}[thm]{Corollary}
 \newtheorem{lem}[thm]{Lemma}
 \newtheorem{prop}[thm]{Proposition}
 \theoremstyle{definition}
 \newtheorem{defn}[thm]{Definition}
 \theoremstyle{remark}
 \newtheorem{rem}[thm]{Remark}
 
 \newtheorem{notat}[thm]{Notation}

\let\oldmarginpar\marginpar
  \renewcommand\marginpar[1]
  {\-\oldmarginpar[\raggedleft\footnotesize\ \textit{#1}]%
  {\raggedright\footnotesize \textit{#1}}}

\begin{document}

\title[Vector-Valued Multiplicative Functions II]
{Free Group Representations from \\  Vector-Valued Multiplicative Functions, II}
\author{M. Gabriella Kuhn}
\address{ Dipartimento di Matematica\\ Universit\`{a} di Milano ``Bicocca'', Edificio U5 Via Cozzi 53, 20125 Milano, ITALIA }
\curraddr{}
\email{mariagabriella.kuhn@unimib.it}
\thanks{}

\author{Sandra Saliani}
\address{Dipartimento di Matematica, Informatica ed Economia\\ Universit\`{a} degli Studi
 della Basilicata, Viale dell'ateneo lucano 10, 85100 Po\-ten\-za, ITALIA}
\curraddr{}
\email{sandra.saliani@unibas.it}
\thanks{}

\author{Tim Steger}
\address{Struttura di Matematica e Fisica \\ Universit\`{a} degli Studi di Sassari,
Via Vienna 2, 07100 Sassari, ITALIA}
\curraddr{}
\email{steger@uniss.it}
\thanks{}

\subjclass{Primary: 43A65, 43A35. Secondary: 15A42, 15B48, 22D25, 22D10}

\keywords{irreducible unitary representations; free groups;
boundary realization; summabilty methods for matrix coefficients}

\date{\today}

\dedicatory{}

\begin{abstract}
Let $\Gamma$ be a non-commutative free group on finitely many generators.
In a previous work two of the authors have constructed the class of multiplicative  representations of $\Gamma$ and proved them irreducible as
 representation of $\G\ltimes_\lambda C(\Om)$.
In this paper we  analyze multiplicative representations
as  representations of $\Gamma$ and we prove a criterium for
irreducibility based on  the growth of their
matrix coefficients.
\end{abstract}

\maketitle

\section{Introduction}

Let $\Gamma$ be a non-commutative free group on finitely many generators,
$\Om$ its boundary and $C(\Om)$ the $C^*$-algebra of complex valued continuous functions on $\Om$.
We say that a unitary representation of $\Gamma$ is {\it tempered} if it is weakly contained in the regular
 representation or, alternatively, if it is a representation of
$C^*_{\text{red}}(\Gamma)$, the regular $C^*$-algebra of $\G$.

In \cite{K-S3}, the first of a series of papers, two of the authors have constructed the class of {\it multiplicative } representations: they are acting on the completion of some space $\mcH^\infty$ of ``smooth functions'', which is built up
from a {\em matrix system with inner product} denoted by $(V_a,H_{b a},B_a)$.
This class
is large enough to include all tempered representations of $\G$ hitherto constructed using the action of $\Gamma$ on its Cayley graph.
These representations are easily extendable to {\it boundary representations},
that is representations of the crossed product $C^*$-algebra
$\Gamma\ltimes_\lambda C(\Omega)$.
In \cite{K-S3} it has been proved that multiplicative representations are
irreducible  when considered as  boundary
 representations, and  criteria have been given to say exactly when two
of them are equivalent.

In this paper we give conditions that ensure the irreducibility of a boundary
representation {\it as a representation of $\Gamma$.}

Our  criteria are based on general facts concerning {\it boundary realizations}
\cite{K-S2} as well as  on the computation of the growth of matrix coefficients.

 In short, a
{\it boundary realization}  of a  unitary representation $(\pi,\mathcal{H})$ of $\Gamma$
is a pair
$(\iota,\pi')$ where
\begin{itemize}
\item $\pi'$~is a representation of $\Gamma\ltimes_\lambda C(\Omega)$
on a Hilbert space~$\mathcal{H}'$;
\item $\iota$~is an isometric $\G$-inclusion of $\mathcal{H}$ into $\mathcal{H}'$;
\item $\mathcal{H}'$ is generated as a $(\G,C(\Om))$-space
by $\iota(\mathcal{H})$.
\end{itemize}

If $\iota$ is unitary (i.e. $\mathcal{H}'=\mathcal{H}$),
the boundary realization is
called {\it perfect} otherwise we shall say that $\iota$ is {\it imperfect}.

Since $\G$ acts amenably (in the sense of Zimmer) on $\Omega$, a
representation $(\pi,\mathcal{H})$ of $\G$ admits a boundary realization if and only
if it is tempered. This follows from the general considerations in \cite{Q-S};
a short proof specifically for the case at hand can be found in \cite{I-K-S}.

Every multiplicative representation $\pi$ provides a boundary realization
of itself when considered as a representation of
$\Gamma\ltimes_\lambda C(\Omega)$: are there other
boundary realizations? In this paper we give a criterion, based on the growth
of matrix coefficients, that ensures that there are {\it no other} boundary
realizations.

Let us briefly explain our main tools.
In 1979 Haagerup \cite{H} showed
 that, for a representation
$\pi$ of $\Gamma$ having a cyclic vector $v$,
the following conditions are equivalent:
\begin{itemize}
\item[i)] $\pi$ is tempered;
\item[ii)] The function $\phi_\eps^v(x)= < v,\pi(x)v> e^{-\eps|x|}$
is square integrable for every positive $\eps;$
\item[iii)]
$\displaystyle
\sum_{|x|=n} |< v, \pi(x)v>|^2\leq  (n+1)^2\|v\|^4.$
\end{itemize}
A consequence of iii) is
\begin{equation}\label{haag}
\|\phi_\eps^v\|_2^2=
\sum_{x\in\Gamma} |< v,\pi(x)v>|^2e^{-2\eps|x|}
\leq C\|v\|^4 \left(\frac1{\eps}\right)^3
\period
\end{equation}
We shall
write $\|\phi_\eps^v\|_2^2\simeq\frac1{\epsilon^\alpha}$ if there exist positive
constants $c_1$ and $c_2$, possibly depending on $v$, such that
\begin{equation*}
\frac{c_1}{\epsilon^\alpha}\leq\|\phi_\eps^v\|_2^2\leq \frac{c_2}{\epsilon^\alpha}.
\end{equation*}
The exponent $3$ for $1/\eps$ in \eqref{haag}
 is an upper bound for the growth of the $\ell^2$ norm
of $\phi_\eps^v$ which, as far as we know,
 is attained only in very special cases, namely
for the representations corresponding to the
endpoints of the  isotropic/anisotropic principal
series of Fig\`a-Talamanca and Picardello \cite{FT-P},
Fig\`a-Talamanca and Steger \cite{FT-S} while
 for the endpoint representation of the series considered by Paschke
\cite{P}, \cite{P2}, one gets $1/\epsilon^2$.

In this paper we shall produce a method to compute
$\|\phi_\eps^v\|_2^2$ for  a
 multiplicative representation
and we continue
the investigation between the existence of other boundary realizations,
the irreducibility and the behavior
of $\|\phi_\eps^v\|_2^2$ started in \cite{K-S2}.

The main results are the following
\begin{theorem_intro}\label{teorema1}
Given a multiplicative representation $\pi$,
one can always find a positive integer $\alpha\leq3$, depending only on $\pi$,
such that
$\|\phi_\eps^v\|_2^2
\simeq \left(\frac1\epsilon\right)^\alpha$
for  all smooth vectors $v$ in  $\mcH^\infty$.
\end{theorem_intro}

\begin{theorem_intro}\label{main}
Let $\pi$ be a multiplicative representation.
Assume  that for all $v\in \mcH^\infty$
$$
\text{either}\quad\|\phi_\eps^v\|_2^2\simeq\frac1{\varepsilon^2}\quad \text{or}\quad
\|\phi_\eps^v\|_2^2\simeq\frac1{\varepsilon^3}\quad \text{hold as}\quad \eps\rightarrow 0,
$$
 then
\begin{itemize}
\item There is only one boundary realization of $\pi$.
\item $\pi$ is irreducible as a $\Gamma$-representation.
\end{itemize}
\end{theorem_intro}

Finally
we shall provide a necessary and  sufficient condition (see
Lemma \ref{lemma5.4})
under which
$$
\|\phi_\eps^v\|_2^2
\simeq\frac1{\varepsilon^2}\quad \text{ as}\quad \eps\rightarrow 0,
$$
for all vectors $v\in\mathcal{H}^\infty$.

\section{Boundary Representations}

Let $\Gamma$ be a free group on a finite symmetric set of generators $A$.
We shall always use the letters $a$, $b$, $c$, $d$ for elements of $A$.
The identity element is denoted by $e$. Every element has a unique reduced
 expression as $x=a_1\dots a_n$ where $a_j a_{j+1}\neq e$. In this case
the length, $|x|$,  of $x$ is $n$.
  The Cayley graph of $\Gamma$ has as vertex set the elements of $\Gamma$ and
  as undirected edges the couples $\{ x,xa\}$ for $x\in\Gamma$ and $a\in A.$
  The distance between two vertices of the tree is defined as the number of the
  edges joining them, so $d(e,x)=|x|$ and $d(x,y)=|x^{-1}y|.$ Two vertices
 $x_1,x_2,$ of the tree are said adjacent if $d(x_1,x_2)=1.$

The boundary $\Omega$ of $\Gamma$ consists of the set of infinite reduced
words $a_1a_2a_3\dots.$, with the topology defined by  the basis
$$\Omega(x)=\{\omega\in\Omega,\;\text{the reduced word for}\;
 \omega\;\text{starts with}\;x\}\period$$
 The sets $\Omega(x)$
are both closed and open and $\Omega$ is a compact Hausdorff
space homeomorphic to the
Cantor set.
$\Gamma$ acts on  itself by left translation. This action preserves  the
tree structure and extends to an action on the boundary of the tree $\Omega$
by the obvious multiplication by finite and infinite reduced words.
Let ${C}(\Omega)$ be the ${C}^{\ast}$-algebra of continuous
 complex valued functions on $\Omega$, under pointwise operations.
Let $\lambda:\Gamma\rightarrow Aut({C}(\Omega))$ be the action by
left translation
$$(\lambda(x)F)(\omega)=F(x^{-1}\omega).$$
\begin{defn} A boundary representation is a triple
$(\pi_{\Gamma},\pi_{\Omega},\!\mathcal{H})$ where
\begin{itemize}
\item $\pi_{\Omega}:C(\Omega)\rightarrow \mathcal{B}(\mathcal{H})$ is a
 $\ast$-representation of $C(\Omega);$
\item $\pi_{\Gamma}:\Gamma\rightarrow \mathcal{U}(\mathcal{H})$ is
a  unitary representation of $\Gamma;$
\item For all $x\in\Gamma$ and $F\in C(\Omega),$
$$\pi_{\Gamma}(x)\pi_{\Omega}(F)\pi_{\Gamma}(x^{-1})=\pi_{\Omega}(\lambda(x)F).$$
\end{itemize}
\end{defn}
Whenever there is no confusion we shall omit the subscripts and write $\pi$
for both $\pi_{\Gamma}$ and $\pi_{\Omega}.$ A boundary representation is nothing
else that a representation of  $\Gamma\;\ltimes_{\lambda} C(\Omega).$
\begin{defn}
A subrepresentation of a boundary representation $\pi$ on $\mathcal{H}$
is a closed subspace of $\mathcal{H}$ invariant under the (restricted)
action of both $\pi(\Gamma)$ and $\pi(C(\Omega)).$

A boundary representation $\pi$ is  irreducible  if
$\mathcal{H}\neq 0$ and $0$ and $\mathcal{H}$ are the only subrepresentations of $\pi.$

Given another boundary representation $\pi^{\sharp}$ on $\mathcal{H}^{\sharp}$,
a unitary map $\mathcal{J}:\mathcal{H}\rightarrow \mathcal{H}^{\sharp}$ such that
$\pi^{\sharp}(x)\mathcal{J}=\mathcal{J}\pi(x),$ for all $x\in\Gamma,$ and
$\pi^{\sharp}(F)\mathcal{J}=\mathcal{J}\pi(F),$ for all $F\in C(\Omega),$ is
called an intertwiner from $\pi$ to $\pi^{\sharp}.$ Two boundary representations
 are called equivalent if there exists an intertwiner between them.
\end{defn}
\subsection{General Results on Boundary Realizations}
\begin{defn}
Given a representation
$(\pi,\mathcal{H})$, we say that a non-zero vector $w\in \mathcal{H}$
satisfies the Good Vector Bound if
there exists a constant $C$, depending only on $w$, such that
\begin{equation}
\tag{GVB}\label{GVB}
\displaystyle\sum_{|x|=n}{|<v,\pi(x)w>|^2}\leq C\| v\|^2,\quad
\text{for all }\; n\in\N, v\in \mathcal{H}.
\end{equation}
\end{defn}
\begin{rem}
We recall that, if $w$ satisfies \eqref{GVB}
and we define, for $\eps>0$, $\phi^w_{\eps}(x)=<w,\pi(x)w> e^{-\eps|x|}$, the growth condition  \eqref{haag} discussed in the
Introduction becomes
$$
\|\phi_\eps^w\|_2^2 =
\sum_{n=0}^\infty\sum_{|x|=n}|< w,\pi(x)w>|^2e^{-2\eps|x|}\leq
\frac{C\|w\|^2}{1-e^{-2\eps}}\simeq\frac1\eps
$$
telling us that the quantity $\eps\|\phi^w_{\eps}\|_2^2$ is bounded as $\eps\to0$.
\end{rem}

There is a deep relation between the existence
of imperfect boundary realizations and the magnitude of the quantities
$\norm{\phi_\eps^v}^2_2$, namely we have the following
\begin{prop}\cite{K-S2}
If a representation $(\pi,\mathcal{H})$ of $\G$ admits an {\it imperfect}
boundary realization, then some nonzero vector $w\in \mathcal{H}$ satisfies
\eqref{GVB}.
\end{prop}
\begin{coro}\label{sistemiinequiv}\cite{K-S2}
Let $(\pi,\mathcal{H})$ be a unitary representation of $\G$ and suppose that
no nonzero vector satisfies \eqref{GVB}. Then any two boundary realizations
of $\pi$ are equivalent.
\end{coro}
\begin{coro}\label{chiave}\cite{K-S2}
Let $(\pi,\mathcal{H})$ be a unitary representation of $\G$ and suppose that
no nonzero vector satisfies \eqref{GVB}. Let $(\pi',\mathcal{H}')$
 be a boundary realization
of $\pi$ and assume that $\pi'$ is irreducible as a representation
of $\Gamma\;\ltimes_{\lambda} C(\Omega)$.
Then $\pi$ is irreducible as
a representation of $\G$.
\end{coro}

\section{Multiplicative Representations, Irreducibility and Inequivalence}

 A matrix system (system in short) $(V_a, H_{ba})$ consists of
finite dimensional complex vector spaces $V_a,$ for each $a\in A,$ and linear maps
$H_{ba}:V_a\rightarrow V_b$ for each pair $a,b\in A,$ where $H_{ba}=0$
whenever $ba=e.$
\begin{defn}
An invariant subsystem of $(V_a, H_{ba})$ is a collection of subspaces
$W_a\subseteq V_a$ such that $H_{ba}(W_a)\subseteq W_b,$ for all $a,b\in A.$

The system $(V_a, H_{ba})$ is called {\bf irreducible} if it is nonzero and there are no
invariant subsystems except for itself and the zero subsystem.
\end{defn}
\begin{defn}
A map from the system $(V_a, H_{ba})$ to
$(V_a^{\sharp}, H_{ba}^{\sharp}),$ is a tuple of linear maps $(J_a),$ where
 $J_a:V_a\rightarrow V_a^{\sharp},$ and $H_{ba}^{\sharp} J_a=J_b H_{ba}.$
 The map $(J_a)$ is called an equivalence if each $J_a$ is a bijection, in that case
 the systems are called equivalent.
\end{defn}
\begin{rem}\label{inequivalenti}
A map $(J_a)$ between irreducible systems
$(V_a, H_{ba})$ and
$(V_a^{\sharp}, H_{ba}^{\sharp}),$ is either $0$ or an equivalence. This is because
the kernels (respectively the images)
 of the maps $J_a$ constitute an invariant subsystem.
\end{rem}
\begin{defn}
The triple  $(V_a,H_{ba},B_a)$ is a system with inner products
if $(V_a,H_{ba})$ is a matrix system, $B_a$ is a positive definite sesquilinear form
on $V_a$ for each $a\in A,$ and  for any $a\in A$ and $v\in V_a$ one has
\begin{equation}\label{compatibility}
B_a(v,v)=\sum_{b\in A}{B_b( H_{ba}v, H_{ba}v) }.
\end{equation}
\end{defn}

Every irreducible matrix system can be normalized so that it admits
 a unique (up to scalars)
 tuple $(B_a)$ of {\it strictly} positive definite forms
(see \cite{K-S3} Theorem 4.9).
 From this point on all the systems that we shall consider will be both
{\bf irreducible and normalized} so that \eqref{compatibility} holds for a given tuple of positive definite forms.

\begin{defn}
Let $(V_a, H_{ba}, B_a)$ be an irreducible  system with inner products. A multiplicative function
is a map
$f:\Gamma\rightarrow \amalg_{a\in A}{V_a}$ satisfying the following condition: there exists
$N=N(f)$ such that for any $x\in\Gamma,$ $|x|\geq N,$
\begin{equation}\label{molti}\begin{array}{ll}
f(xa)\in\ V_a,&\text{if}\; |xa|=|x|+1,\\
f(xab)=H_{ba}f(xa),& \text{if}\;|xab|=|x|+2.
\end{array}
\end{equation}

Two multiplicative functions $f,g$ are called equivalent if $f(x)=g(x)$ for all but finitely many
elements of $\Gamma$. $\mathcal{H}^{\infty}$
denotes the quotient space of the space
of multiplicative functions with respect to this equivalence relation.
For any $f_1,f_2\in \mathcal{H}^{\infty}$ let
\begin{equation}\label{inner}
<f_1,f_2> 
=\sum_{|x|=N}\sum_{\atopn{a\in A}{|xa|=|x|+1}}
{B_a(f_1(xa),f_2(xa))},
\end{equation}
where $N$ is big enough so that both $f_1$ and $f_2$ satisfy (\ref{molti}).
\end{defn}
\begin{defn}
The  completion of  $\mathcal{H}^{\infty}$ with respect to the norm induced by the inner product
(\eqref{inner}) will be our  representation space $\mathcal{H}$.
\end{defn}

Multiplicative functions can also be defined starting from matrix systems which
are not irreducible (see \cite{Iozzi_Kuhn_Steger_stab}).
 In this case one can
still find a tuple of {\it positive semidefinite forms $B_a$}
such that \eqref{compatibility} holds. Then one can proceed to define an
inner product as in \eqref{inner}. However in this case the inner product
\eqref{inner} will
induce a seminorm and $\mcH$ will split into the direct sum of orthogonal
(with respect to to the $B_a$) subspaces
(see \cite{Iozzi_Kuhn_Steger_stab} Section 5). As a consequence
 the corresponding
multiplicative
representation will be reducible and we shall not consider this possibility.

For any directed edge $(x,xa)$ of the tree, we define
$$\Gamma(x,xa)=\{y\in\Gamma,\;d(y,xa)<d(y,x)\},$$
and we get $\Gamma=\Gamma(x,xa)\amalg\Gamma(xa,x).$
We set also  $\Gamma(a)=\Gamma(e,a),$ and
$$
\G(x)=\{z\in\G,\; \mbox{the reduced word for $z$ starts with $x$}\}
$$
$$
\widetilde{\Gamma}(a)=\{y\in\Gamma,\;\text{the reduced word for}\; y\;
\text{ends in}\;a\}\period
$$

The following functions can be considered, quite rightly,
the bricks at the base of multiplicative functions.
\begin{defn}
 For a fixed $x\in\Gamma,$
$a\in A,$ and $v_a\in V_a,$ let
$\mu[x,x a, v_a]:\Gamma\rightarrow \amalg_{b\in A}{V_b}$ be as follows
\begin{itemize}
\item $\mu[x,x a, v_a](y)=0,$ for $y\neq\Gamma(x,xa);$
\item $\mu[x,x a, v_a](xa)=v_a;$
\item $\mu[x,x a, v_a](ybc)=H_{cb}\, \mu[x,x a, v_a](yb),$ if $yb,ybc\in\Gamma(x,xa),$ and
$d(ybc,x)=d(y,x)+2.$
\end{itemize}
Note that $y\G(x,xa)=\G(yx,yxa)$ 
and,  modulo the equivalence relation, one has
$\mu[x,xa,v_a](\inv y \cdot)=\mu[yx,yxa,v_a](\cdot)$.

\end{defn}
Let
${\bf 1}_{\Omega(y)}$, respectively ${\bf 1}_{\Gamma(y)},$
 be the characteristic function of the set
$\Omega(y)$, respectively $\Gamma(y)$.
The {\it multiplicative representation $\pi$} will act on $\mathcal{H}^\infty$
according to the rules
$$
\begin{array}{l}
\pi_\Gamma(y)f(x) = f(\inv y x)\\
\pi_\Om(\1_{\Om(y)})f=\1_{\Gamma(y)}f\;.
\end{array}
$$

Observe that, modulo the equivalence relation, one  has
\begin{equation}\label{??}
\mu[yx,yx a, v_a]~=~\pi(y)\mu[x,xa, v_a]\quad\mbox{ for all $y\in\G$}\;,
\end{equation}
irrespective of whether $|xa|=|x|+1$.
In particular, for $c\in A$ and $w\in V_{c^{-1}}$, one has
\begin{equation}\label{befana}
\mu[c,e,w](x)=\pi(c)\mu[e,\inv{c}, w](x)=\left\{\begin{aligned}
& w,\quad \mbox{if $x=e$},\\
& \sum_{a\neq c}\mu[e,a, H_{a { c^{-1}}} w](x),\quad\mbox{if $x\neq e$}.
\end{aligned}\right.
\end{equation}

Fix now $y\in\G$, choose $N>|y|+1$ and write
$$f=\sum_{|x|=N}\sum_{\atopn{a\in A}{|xa|=|x|+1}}{\mu[x,x a,f(xa)]}$$
as an orthogonal sum of elementary multiplicative functions with disjoint
supports. Since
\begin{equation*}
\pi(y)f=
\sum_{|x|=N}\sum_{\atopn{a\in A}{|xa|=|x|+1}}{\mu[yx,yx a,f(xa)]}
\end{equation*}
and the sets $\G(yx,yxa)$ are also all disjoint
\eqref{??} says that $\pi$ is unitary.

Finally, since $C(\Omega)$ is generated by the functions
$\{{\bf 1}_{\Omega(x)},\,x\in\Gamma\},$
it is easy to verify that the pair $(\pi_\Gamma,\pi_\Om)$ extends to a boundary
representation of $\Gamma$ on $\mathcal{H}$
 that we shall simply denote by $\pi$.


\section{Proof of Theorem 2}

\begin{reftheorem}{main}
Let $\pi$ be the multiplicative representation constructed from
an irreducible, normalized matrix system
 $(V_a, H_{ba},B_a)$.
Assume  that for all $v\in \mcH^\infty$
\begin{equation}\label{introbis}
\text{either}\quad\|\phi_\eps^v\|_2^2\simeq\frac1{\varepsilon^2}\quad \text{or}\quad
\|\phi_\eps^v\|_2^2\simeq\frac1{\varepsilon^3}\quad \text{hold as}\quad \eps\rightarrow 0,
\end{equation}
 then
\begin{itemize}
\item There is only one boundary realization of $\pi$,
\item $\pi$ is irreducible as a $\Gamma$-representation.
\end{itemize}
\end{reftheorem}

\begin{proof}
Since $\pi$ is irreducible  as a representation of $\G\ltimes C(\Om)$
(\cite{K-S2} Theorem 5.3) by Corollaries
 \ref{chiave} and \ref{sistemiinequiv}  we only have to prove that
no nonzero $g\in\mathcal{H}$ satisfies the Good Vector Bound.
The structure of the proof is, as in  Lemma 1.6 of \cite{K-S2},
by contradiction.
There are however crucial not straightforward differences in the choice of the
main objects, due to the vector setting .

Assume that there exists a nonzero $g\in\mathcal{H},$
and a constant $C$ depending only on $g$, such that for every
$f\in\mathcal{H},$ and every positive integer $n$ one has
\begin{itemize}
\item[(GVB)]\hspace{2cm}
$\displaystyle\sum_{|x|=n}{|<f,\pi(x)g> |^2}\leq C\| f\|^2.$
\end{itemize}
By linearity we may assume that $\| g\|=1$.

We shall allow the constant $C$ to change from line to line, keeping in mind
that it will always be {\it independent on $n$}.

The condition (GVB) implies that
\begin{equation}\label{gvbbis}
\limsup_{\varepsilon\rightarrow 0^+}\quad\eps\;
\sum_{x\in\Gamma}{|<f,\pi(x)g>|^2 e^{-\varepsilon |x|}}= C\|f\|^2
<+\infty.
\end{equation}
Fix  $a\in A$, $v_a\in V_a$ such that $B_a(v_a,v_a)=1$.

Let
$f=\mu[e,a,v_a],\quad f_y=\1_{\Gamma(y)}f=\pi(\1_{\Omega(y)})f,\; $
and observe that
\begin{equation}\label{sommeortogonali}
f=\sum_{|y|=n}f_y,\quad \text{and } \quad\norm{f}^2=\sum_{|y|=n}\norm{f_y}^2=1,
\end{equation}
since \eqref{sommeortogonali} is an orthogonal sum.

For this given $f$ we shall compute \eqref{gvbbis}. Since a finite number of terms gives a zero contribute to the lim sup, we may rewrite
$$
C=\limsup_{\varepsilon\rightarrow 0^+}
\varepsilon\sum_{|y|=n}\sum_{x\in\Gamma(y)}{|<f,\pi(x)g>|^2 e^{-\varepsilon |x|}}
=\limsup_{\varepsilon\rightarrow 0^+}C^n_f(\varepsilon).
$$
For each $x\in\G(y)$ write $|<f,\pi(x)g>|^2 =|<f_y+f-f_y,\pi(x)g>|^2$ to get
\begin{eqnarray*}
C_f^n(\varepsilon)&=&\varepsilon\sum_{|y|=n}\sum_{x\in\Gamma(y)}
{|<f_y,\pi(x)g>|^2 e^{-\varepsilon |x|}}\\ \\
& &+\varepsilon\sum_{|y|=n}\sum_{x\in\Gamma(y)}{2{\Re}e(<f_y,\pi(x)g>
\overline{<f-f_y,\pi(x)g>}) e^{-\varepsilon |x|}}\\ \\
& &+\varepsilon\sum_{|y|=n}\sum_{x\in\Gamma(y)}{|<f-f_y,\pi(x)g>|^2
 e^{-\varepsilon |x|}}\\ \\
 &=&C_{1,f}^n(\varepsilon)+2C_{2,f}^n(\varepsilon)+C_{3,f}^n(\varepsilon).
\end{eqnarray*}
By Cauchy-Schwarz inequality $|C_{2,f}^n(\eps)|\leq
|C_{1,f}^n(\eps)C_{3,f}^n(\eps)|^\frac12$, moreover,
condition \ref{GVB} and \eqref{sommeortogonali}
imply  that
\begin{eqnarray*}
\limsup_{\varepsilon\rightarrow 0^+}C_{1,f}^n(\varepsilon) &=&
\limsup_{\varepsilon\rightarrow0^+}\;\varepsilon\sum_{|y|=n}\sum_{x\in\Gamma(y)}
{|<f_y,\pi(x)g>|^2 e^{-\varepsilon |x|}}\\
&\leq &
\sum_{|y|=n}\limsup_{\eps\to0^+}\;\varepsilon\;\sum_{x\in\G}|< f_y,\pi(x)g
>|^2e^{-\eps|x|}\\
&\leq&
C\sum_{|y|=n}\|f_y\|^2= C\|f\|^2=C\;.
\end{eqnarray*}
For sufficiently small $\eps$ one has $|C_{1,f}|^n(\eps)\leq 2C,$
so that
 $$0\leq C_{3,f}^n(\eps)\leq C^n_f(\eps)-2C_{2,f}^n(\eps)
\leq C^n_f(\eps)+2\, (2C \,C_{3,f}^n(\eps))^\frac12.
$$
Hence, if $\limsup_{\eps\to0^+}C^n_f(\eps)$ is finite, the following
\begin{equation}\label{a3}
\limsup_{\varepsilon\rightarrow 0^+}C_{3,f}^n(\varepsilon)
\end{equation}
is also finite, say
$\limsup_{\varepsilon\rightarrow 0^+}C_{3,f}^n(\varepsilon)\leq C$.

In the next section we shall prove (see Corollary \ref{coro3.4}) that, since
$\mathrm{supp}(f-f_y)\subset \Gamma\setminus\Gamma(y)$ the above $\limsup$
\eqref{a3} is actually a limit, more precisely we shall prove that
there exists a tuple $\widehat{B}_c$ of strictly positive definite forms
on $\widehat{V}_c$, the space of antilinear functionals on $V_{c^{-1}}$,
such that
\begin{eqnarray}
\lim_{\varepsilon\rightarrow 0^+}C_{3,f}^n(\varepsilon)
 &=&
\sum_{c\in A}\sum_{\atopn{|z|=n-1}{|zc|=n}}\lim_{\varepsilon\rightarrow 0^+}\varepsilon
\sum_{x\in\Gamma(zc)}{|<f-f_{zc},\pi(x)g>|^2 e^{-\varepsilon |x|}}
\nonumber\\ \nonumber\\
&=&
\sum_{c\in A}\sum_{\atopn{|z|=n-1}{|zc|=n}}
\frac{1}{k_0}\,\widehat{B}_c(S\pi(z^{-1})(f-f_{zc}),S\pi(z^{-1})(f-f_{zc})),
\label{bc}
\end{eqnarray}
where
$S\pi(z^{-1})(f-f_{zc})$
 is the antilinear
functional on ${V}_{c^{-1}}$
defined by the following rule
$$\langle {w}, S\pi(z^{-1})(f-f_{zc})
\rangle
=<\pi(z^{-1})(f-f_{zc}),\mu[c,e,w]>,$$
for every $w\in{V}_{c^{-1}}$ and $\mu[c,e,w]$ is as in \eqref{befana}.
Hence $S\pi(z^{-1})(f-f_{zc})$ will be identified
with an element of $\widehat{V}_c=\overline{V}_{c^{-1}}'$
(the reader may refer to Subsection \ref{stwin} and Definition \ref{twin}).

For the moment we shall assume valid \eqref{bc} and we proceed
with the calculations.
Let $u\in \widehat{V}_c=\overline{V}_{c^{-1}}'$ and let
$\|u\|_\infty =\displaystyle{\sup_{\atopn{v \in {V}_{c^{-1}}}{B_{c^{-1}}(v,v)=1}}{|\langle{v},u\rangle|}}$
denote its
norm as an antilinear functional on $V_{c^{-1}}$.
Since $u\in \widehat{V}_c\mapsto\widehat{B}_c(u,u)^{1/2}$
also
 defines a norm on the same finite dimensional Banach space
  $\widehat{V}_c,$
%
there exists a positive constant $K_c$, depending only on $c$ and on
$\widehat{B}_c$,
such that for any
$u\in  \widehat{V}_c$ and for any unit vector $w\in V_{c^{-1}}$
$$K_c |\langle{w},u\rangle|\leq K_c\|u\|_\infty \leq
 \widehat{B}_c(u,u)^{1/2}\;.$$

This yields a below estimate for each term in the sum \eqref{bc} above:
for any $c\in A$,   $w\in V_{c^{-1}},$ {$B_{c^{-1}}(w,w)=1$} and
$z$ such that  $|zc|=|z|+1,$
one has
\begin{eqnarray}\label{bbc}
 & &\widehat{B}_c(S\pi(z^{-1})(f-f_{zc}),S\pi(z^{-1})(f-f_{zc})) \\
 &&\hspace{4cm}
 \geq k_c\,|<\pi(z^{-1})(f-f_{zc}),\mu[c,e,w]>|^2 \nonumber
 \end{eqnarray}
where $k_c=K_c^2>0$.
Putting together \eqref{a3}, \eqref{bc}, and  \eqref{bbc} we get the following

{\bf Claim.}

For every $c\in A$ there exists a positive constant
 $k_c$ such that, for any $n\in\N$ and
for any choice of vectors
$w_{c^{-1}}\in V_{c^{-1}},$ {$B_{c^{-1}}(w_{c^{-1}},w_{c^{-1}})=1$},
$$
B_0^n(f):=\sum_{c\in A}\sum_{\atopn{|z|=n-1}{|zc|=n}}
k_c|<\pi(z^{-1})(f-f_{zc}),\mu[c,e,w_{c^{-1}}]>|^2
$$
 is uniformly bounded in $n$.
\vspace{0.5cm}\\
We fix unit vectors $w_{c^{-1}}\in V_{c^{-1}}$,
with $w_a=v_a$ and we write, as we did for $C^n_f(\eps)$,  $B_0^n(f)$
as the sum of three terms:
$$
B_0^n(f)=B_{1,f}^n-2B_{2,f}^n+B_{3,f}^n
$$
where
\begin{eqnarray*}
& &\!\!\!\!B_{1,f}^n  =\!\!
\sum_{c\in A}\sum_{\atopn{|z|=n-1}{|zc|=n}}
k_c|<f_{zc},\pi(z)\mu[c,e,w_{c^{-1}}]>|^2\\ \\
& &\!\!\!\!B_{2,f}^n = \\
& &\!\!\!\!
\sum_{c\in A}\sum_{\atopn{|z|=n-1}{|zc|=n}}\!\!\!
k_c\, 2{\Re}e(<f,\pi(z)\mu[c,e,w_{c^{-1}}]>\!
\overline{<f_{zc},\pi(z)\mu[c,e,w_{c^{-1}}]>})\\ \\
& &\!\!\!\!B_{3,f}^n = \!\!\!\!
\sum_{c\in A}\sum_{\atopn{|z|=n-1}{|zc|=n}}
k_c|<f,\pi(z)\mu[c,e,w_{c^{-1}}]>|^2.
\end{eqnarray*}

We use again \eqref{sommeortogonali} to estimate $B_{1,f}^n$:
\begin{eqnarray*}
B_{1,f}^n
&=&
\sum_{c\in A}\sum_{\atopn{|z|=n-1}{|zc|=n}}
k_c|<f_{zc},\pi(z)\mu[c,e,w_{c^{-1}}]>|^2\\ \\
&\leq&
 \sum_{c\in A}k_c\sum_{\atopn{|z|=n-1}{|zc|=n}}\|f_{zc}\|^2
\|\pi(z)\mu[c,e,w_{c^{-1}}]\|^2\\ \\
&\leq&
k_0\sum_{c\in A}\sum_{\atopn{|z|=n-1}{|zc|=n}}\|f_{zc}\|^2=k_0\|f\|^2,
\end{eqnarray*}
where $k_0=\max_{c\in A}k_c$.

Arguing as before, we may conclude that
  there exists
 a constant $C_2,$
possibly depending on
$a$, but independent on $n$ such that
 $$B_{3,f}^n \leq C_2.$$


The final step will consist in  showing that the uniform boundedness in $n$ of
$$B_{3,f}^n=\sum_{|z|=n-1}\sum_{\atopn{c\in A} {|zc|=n}}
k_c|<f,\pi(z)\mu[c,e,w_{c^{-1}}]>|^2$$
yields a contradiction.

Recall that $f=\mu[e,a,v_a]$ and
call $M=M(a)$ the upper bound such that, for all $n$,
$$B_{3,f}^n=\sum_{|z|=n-1}\sum_{\atopn{c\in A} {|zc|=n}}
k_c|<f,\pi(z)\mu[c,e,w_{c^{-1}}]>|^2\leq M.$$
Setting $ k_1=\min_{c\in A}{k_c}>0,$ this yields
$$
\sum_{|z|=n-1}\sum_{\atopn{c\in A} {|zc|=n}}|<f,\pi(z)\mu[c,e,w_{c^{-1}}]>|^2\leq \frac{M}{k_1}=M_1.
$$

Write
\begin{eqnarray*}
\sum_{|x|=n}|<f,\pi(x)f\!>|^2\!\!\!&=&\!\!\!
\sum_{\atopn{|x|=n}{x\in\Gamma(a)}}|<f,\pi(x)f\!>|^2+
\sum_{\atopn{|x|=n}{x\notin\Gamma(a)}}|<f,\pi(x)f\!>|^2\\
\!\!\!&=&\!\!\! \text{(I)}+\text{(II)}
\end{eqnarray*}
For  (I) we get
\begin{eqnarray*}
\sum_{\atopn{|x|=n}{x\in\Gamma(a)}}|<f,\pi(x)f>|^2\!\!\!
&=&\!\!\!
\sum_{\atopn{|z|=n-1}{ z\notin\Gamma(a^{-1})}}|<\pi(az)f,f>|^2\\
&=&
\sum_{\atopn{|z|=n-1}{ z\notin\widetilde{\Gamma}(a)}}|<f,\pi(z\,a^{-1})f>|^2\\
&=&
\sum_{\atopn{|z|=n-1}{ z\notin\widetilde{\Gamma}(a)}}|<\mu[e,a,v_a],\pi(z)\mu[a^{-1},e,v_a]>|^2\\
&\leq&\!\!
\sum_{|z|=n-1}\sum_{\atopn{c\in A} {|zc|=n}}|<f,\pi(z)\mu[c,e,w_{c^{-1}}]>|^2\!\leq M_1.
\end{eqnarray*}

For  (II) we get
\begin{eqnarray*}
\sum_{\atopn{|x|=n}{x\notin\Gamma(a)}}|<f,\pi(x)f>|^2
&=&
\sum_{\atopn{\atopn{|x|=n}{ x\notin\Gamma(a)}}{|xa|=n+1}}
|<\mu[e,a,v_a],\pi(x)\mu[e,a,v_a]>|^2\\
& &+\sum_{\atopn{\atopn{|x|=n}{x\notin\Gamma(a)}}{ |xa|=n-1}}
|<\mu[e,a,v_a],\pi(x)\mu[e,a,v_a]>|^2.
\end{eqnarray*}
The first sum in (II) is equal to zero since $x\notin\Gamma(a)$
and $xa$ does not reduce.
The second sum is
\begin{eqnarray*}
& &\sum_{\atopn{\atopn{|x|=n}{x\notin\Gamma(a)}}{|xa|=n-1}} |<\mu[e,a,v_a],\pi(x)\mu[e,a,v_a]>|^2\\
& =&
\sum_{\atopn{\atopn{|z|=n-1}{z\notin\Gamma(a)}}{|za^{-1}|=n}}
|<\mu[e,a,v_a],\pi(z)\mu[a^{-1},e,v_a]>|^2\\
&\leq &
\sum_{\atopn{|z|=n-1}{z\notin\Gamma(a)}}
\sum_{\atopn{c\in A}{ |zc|=n}}
|<\mu[e,a,v_a],\pi(z)\mu[c,e,w_{c^-1}]>|^2\leq M_1.
\end{eqnarray*}

Hence there exists a constant $C_3=2M_1>0$ such that
$$\sum_{|x|=n}{|<\mu[e,a,v_a],\pi(x)\mu[e,a,v_a]>|^2}\leq C_3,\quad \mathrm{for \; any}\; n,$$
 and we get a  contradiction since the hypothesis \eqref{introbis} on
$\|\phi_\eps^v\|_2^2$, with the choice $v=\mu[e,a,v_a],$ yields
for any $\varepsilon>0$ that either $\varepsilon^{-2}$ or $\varepsilon^{-3}$
is bounded by
$$ \sum_{n=0}^{+\infty}{\sum_{|x|=n}{\!|\!<\!\mu[e,a,v_a],\pi(x)\mu[e,a,v_a]\!>\!\!|^2}
 \,e^{-2\varepsilon n}}\!\leq\!
C_3
\sum_{n=0}^{+\infty}{ e^{-2\varepsilon n}}
\simeq \frac1\eps\;.
$$
\end{proof}

\section{Computation of Matrix Coefficients}

This section is devoted to the computation of the quantities
$$
\|\phi_\varepsilon^{v_a,v_b}\|^2_2=
\sum_{n=0}^{+\infty}{\sum_{|x|=n}{\!|\!<\!\mu[e,a,v_a],\pi(x)\mu[e,b,v_b]\!>\!\!|^2}
 \,e^{-2\varepsilon n}}.
$$
In Theorem 1 we shall show that these quantities have always polinomial growth
with respect to $1/\eps$.
Finally, in Lemma~\ref{lem3.3} we shall provide the
 exact asymptotics for
$\|\phi_\varepsilon^{v_a,v_b}\|^2_2$
 which are needed to prove Theorem~\ref{main}.

\subsection{The Twin of the System}\label{stwin}

Throughout the whole paper we shall use the following
notation.

If $V_1$ and $V_2$ are finite dimensional complex vector spaces,
$\mathscr{L}(V_1,V_2)$ is the space of linear maps $T:V_1\rightarrow V_2.$
 The dual space of $V$ is $V'=\mathscr{L}(V,\C)$
and the duality is given, as usual,  by $\langle v,v'\rangle$.
If $T\in\mathscr{L}(V_1,V_2),$ then the dual map is
${T'}\in\mathscr{L}(V_2',V_1')$.

 $\overline{V}$ is the complex conjugate vector space of $V$, i.e. the set
$V$ with the same addition operation, but with an altered multiplication
$$\lambda\overline{\cdot} v=\overline{\lambda}v,\quad \lambda\in\C,\;v\in V.$$
A  map  $f:V\rightarrow\C$ is called ``antilinear" if
$f:\overline{V}\rightarrow\C,$ is linear.


 The space of antilinear functionals on $V$ is denoted by $V^{\ast}=\overline{V}'.$

We recall some identifications that will be used repeatedly in this paper.

Given finite dimensional complex vector spaces $V_1$ and $V_2,$
for any $v_1\in V_1$, $v_2\in V_2$ and $f\in V_2'$ we consider the map
$$v_1\otimes v_2:V_2'\rightarrow V_1,\quad
(v_1\otimes v_2)(f)=f(v_2)\, v_1=\,\langle v_2,f\rangle\, v_1\;.$$
 Then  $v_1\otimes v_2\in\mathscr{L}(V_2',V_1).$
 By linearity the above extends to an isomorphism
$V_1\otimes V_2\cong\mathscr{L}(V_2',V_1).$

It follows that, given $T_1\in\mathscr{L}(V_1,V_3)$ and
$T_2\in\mathscr{L}(V_2,V_4)$, the map $$T_1\otimes T_2:V_1\otimes V_2\rightarrow V_3\otimes V_4$$
corresponds to the operator
$$
 \mathscr{L}(V_2',V_1)\rightarrow \mathscr{L}(V_4',V_3),
\quad S\mapsto T_1\,S\,T_2'.
$$
So we shall write
$$
(T_1\otimes T_2)S=T_1S T_2'.
$$
%


The duality isomorphism
$${\mathcal L}:\mathscr{L}(V,V^{\ast})\rightarrow\mathscr{L}(V^{\ast},V)'$$
defines a bilinear form which can be written explicitly by means of the trace function
$$B: \mathscr{L}(V,V^{\ast})\times\mathscr{L}(V^{\ast},V)\rightarrow \C,$$
\begin{equation}\label{otto}
B(T,S):=({\mathcal L}(T))(S)=\mathrm{tr}(TS)=\mathrm{tr}(ST).
\end{equation}
In particular, when
$T=(\bar v_1'\otimes v_2')$ and $S=(v_3\otimes\bar v_4)$
($v_3,v_4\in V$, $v_1',v'_2\in V'$) are elementary tensors, one has
\begin{equation}\label{nove}
\mathrm{tr}\left( (\bar v_1'\otimes v_2')(v_3\otimes\bar v_4)\right)=
\overline{\langle v_4,v'_1}\rangle\langle v_3,v_2'\rangle\;.
\end{equation}
In this case we shall omit the $\tr$ in front and we shall write, for
brevity, $(\bar v_1'\otimes v_2')(v_3\otimes\bar v_4)$.

Positive definite sesquilinear forms $B_a$
on the space $V_a$ are identified with
maps $B_a \in\mathscr{L}(V_a,V_a^{\ast}),$ via the linear extension
$$(B_a (\lambda v))(\mu  w)=B_a(\lambda v, \mu w)=\lambda \overline{\mu}B_a(v,w)=
\lambda \overline{\mu}B_a(v)(w)$$
for any $v,w\in V_a,$ and $\lambda,\mu\in\C$.
Under this identification one also has $B^*_a=B_a$.

For every $a\in A$ set
$$
\widehat{V}_a:=V^{\ast}_{a^{-1}}=\overline{V}_{a^{-1}}'.
$$
Given a  matrix system $(V_a, H_{ba}),$  $H_{ba}$ induces an obvious linear
 map on the space of antilinear functionals on $V_a$, $V_a^{\ast},$  by
$H^{\ast}_{ba}:V_b^{\ast}\rightarrow V_a^{\ast},$
$H^{\ast}_{ba}(f)=f\circ H_{ba},$ and also maps
$$
\widehat{H}_{ba}:=H^{\ast}_{a^{-1}b^{-1}}:\widehat{V}_a\rightarrow\widehat{V}_b,
$$

Hence the matrix system $(V_a, H_{ba})$ induces
another matrix system $(\widehat{V}_a, \widehat{H}_{ba}),$ which is irreducible if
so is $(V_a, H_{ba}).$
\begin{defn}
If for all $a\in A$, the bilinear form
$B_a$ is strictly positive definite,  we shall say for brevity that
{\it the tuple $(B)_a$ is positive definite}. If the $B_a$
are all positive semidefinite and there exists an index $a\in A$ such that
$B_a$ {\it is not positive definite} we shall say that the tuple is positive
semidefinite.
Analogously we define  negative definite tuples.
\end{defn}
\begin{prop}\label{der}
Assume that  $(V_a, H_{ba}, B_a)$ is an irreducible  system
with inner product. Then there exists a unique (up to multiple
scalars) positive definite tuple $(\widehat{B}_a),$
$\widehat{B}_a:\widehat{V}_a\rightarrow\widehat{V}_a^{\ast}$
on $\widehat{V}_a$
such that the matrix system
$(\widehat{V}_a,\widehat{H}_{ba},\widehat{B}_a)$ is
an irreducible system with inner products.
\end{prop}
\begin{proof}
Let $\mcV=\oplus_{a\in A}\;V_a^\ast\otimes V_a'$ and define
$T:\mcV\to\mcV$ by the rule
$$
(TC)_a=\sum_{b\in A} T_{ab}C_b,\quad T_{ab}=H^\ast_{ba}\otimes H_{ba}'\;.
$$
Since every $B_a$ may be regarded as an element of $V_a^\ast\otimes V_a$
the compatibility condition \eqref{compatibility} can be rewritten as
$$
(TB)_a=
\sum_{b\in A} H^\ast_{ba}\otimes H_{ba}'B_b
=\sum_{b\in A}H^\ast_{ba}B_bH_{ba}=B_a.
$$
The above equation says that the tuple $(B_a)$
is a right eigenvector for
the matrix $T=(H^\ast_{ba}\otimes H_{ba}')_{a,b}$
corresponding to eigenvalue $1$.
Hence $1$ is an  eigenvalue  for the transpose matrix
$T'=\left({{H}}_{ab}\otimes \overline{{H}}_{ab}\right)_{a,b}$ too.
To say that $\widehat{B}_a$ is a compatible tuple for
$(\widehat{V}_a,\widehat{H}_{ba})$ is equivalent to say that
$\widehat{B}_a$ is a right eigenvector for the matrix
$\widehat{T}=\left(\widehat{H}_{ba}^{\ast}\otimes {\widehat{H}}_{ba}'\right)_{a,b}
\!=\!\left({H}_{a^{-1}b^{-1}}\otimes \overline{{H}}_{a^{-1}b^{-1}}\right)_{a,b}
$.

But the last matrix is obtained from
$$
\left({{H}}_{ab}\otimes\overline{{H}}_{ab}\right)_{a,b}=T'$$
by interchanges of rows and columns, so $\widehat{T}$ and $T$ have the same
eigenvalues.
Since the matrix system $({V}_a,{H}_{ba})$ is irreducible, then
 $(\widehat{V}_a,\widehat{H}_{ba})$ is irreducible too.
Corollary 4.8 of \cite{K-S3} ensures that there exists an essentially unique
eigentuple $(\widehat{B}_a)$ of strictly  positive definite forms satisfying
\begin{equation*}
\w{B}_a=\sum_{b\in A}\w{H}_{ba}^\ast \w{B}_b\w{H}_{ba}\period
\end{equation*}
\end{proof}

\begin{defn}\label{twin}
We shall call
$(\widehat{V}_a,\widehat{H}_{ba},\widehat{B}_a)$ the
{\it twin system} induced by
$({V}_a,{H}_{ba},{B}_a)$.
\end{defn}

\begin{defn}\label{defE}
For any $a,b\in A,$ we define  maps $E_{ab}:V_b\rightarrow \widehat{V}_a$ by
\begin{equation}\label{ee}
E_{ab}=\sum_{\atopn{c\in A}{ c\neq a, b^{-1}}}{{H}^{\ast}_{c a^{-1}}{B}_c{H}_{cb}}=
\sum_{\atopn{c\in A}{ c\neq a, b^{-1}}}{\widehat{H}_{a c^{-1}}{B}_c{H}_{cb}},
\end{equation}
where $E_{ab}=0$ whenever $ab=e$.

We use also the following notation, for vectors $v_a\in V_a$ and $v_b\in V_b,$
$$E_{a_1 e}:=H_{a a_1^{-1}}^{\ast}B_{a}(v_{a})\in \widehat{V}_{a_1},$$
$$E_{e a_J}:=\overline{H_{b a_J}^{\ast}B_{b}(v_b)}\in V_{a_J}'.$$
\end{defn}

It holds $E_{ab}^{\ast}=E_{b^{-1}a^{-1}}.$ Indeed by taking adjoint
$$E_{ab}^{\ast}=\sum_{\atopn{c\in A}{ c\neq a, b^{-1}}}{{H}^{\ast}_{c b}{B}_c^{\ast}{H}_{ca^{-1}}}=
\sum_{\atopn{c\in A}{ c\neq a, b^{-1}}}{{H}^{\ast}_{c b}{B}_c{H}_{c a^{-1}}}=E_{b^{-1}a^{-1}}.$$

Let $a,b,c,d\in A,$  $v_a\in V_a$ and $v_b\in V_b,$ $J\geq 1.$
We look for a transition matrix which rules the expression
$$
\sum_{\atopn{x\in\Gamma(c)\cap\widetilde{\Gamma}(d)}{|x|=J}}
{|<\mu[e,a,v_a],\pi(x)\mu[e,b,v_b]>|^2}.
$$
It turns out that this matrix
$\mathcal{D}=( D)_{i,j=1\dots 4}$
is a $4\times 4$ block triangular matrix
 obtained
as tensor product of the matrix
\begin{equation}\label{matrixtilde}
{\mathcal{\widetilde D}}=\left(\begin{array}{cc}
(\widehat{H}_{ab})_{a,b} &(E_{ab})_{a,b}\\ \\
0&({H}_{ab})_{a,b}
\end{array}\right)
\end{equation}
by its conjugate, i.e. ${\mathcal{D}}={\mathcal{\widetilde D}}\otimes
\overline{\mathcal{\widetilde D}}.$

Note the following notation, used throughout:
\begin{align*}
\delta(ab)&=\delta_e(ab)=\left\{\begin{aligned}
& 1\quad\mbox{if }\; b=\inv a,\\
& 0\quad\mbox{if }\; ab\neq e.
\end{aligned}\right.
\end{align*}

\begin{lem}\label{mum0}
Let $a,b\in A,$  $v_a\in V_a,$ and $v_b\in V_b.$
For $J\geq 1,$ and  a reduced word $x=a_1 a_2\dots a_J$
we have
\begin{eqnarray}
& &<\mu[e,a,v_a],\pi(a_1\dots a_J)\mu[e,b,v_b]>=\nonumber\\ \nonumber\\
&&\!\!\!
\left(\!\!\begin{array}{c}
(v_b\delta(a_Jb)\delta(a_J^{-1}d'))_{d'\in A}\\ \\
(E_{e a_J}\delta(a_J^{-1}d'))_{d'\in A}
\end{array}\!\right)^{\top}
{\!\mathcal{{\widetilde D}}}
\left(\!\begin{array}{c}
(f^1_{J-1}(a_1\dots a_{J-1})\delta(a_{J-1}^{-1}c'))_{c'\in A}\\   \\
(f^2_{J-1}(a_1\dots a_{J-1})\delta(a_{J-1}^{-1}c'))_{c'\in A}
\end{array}\!\right),  \nonumber\\ \label{mum}
\end{eqnarray}
where  the vectors $f^i_{j-1}(a_1\dots a_{J-1}),$ $i=1,2$, $j=1\dots,J-1,$
are defined recursively as follows
\begin{eqnarray*}
f^1_{1}(a_1)&\!\!\!=&\!\!\!E_{a_1 e}\in \widehat{V}_{a_1},\quad
f^2_{1}(a_1)=v_a\delta(a_1 a^{-1})\in V_{a_1},\\
f^1_{j}(a_1\dots a_{j})&\!\!\!=&\!\!\!\widehat{H}_{a_j a_{j-1}}f^1_{j-1}(a_1\dots a_{j-1})+
E_{a_j a_{j-1}}f^2_{j-1}(a_1\dots a_{j-1})\in\widehat{V}_{a_j}\\
 f^2_{j}(a_1\dots a_{j})&\!\!\!=&\!\!\!H_{a_ja_{j-1}}f^2_{j-1}(a_1\dots a_{j-1})\in V_{a_j}.
\end{eqnarray*}

\end{lem}

\begin{thm}\label{mumu0}
Let $a,b,c,d\in A,$  $v_a\in V_a,$ and $v_b\in V_b.$ Then
\begin{equation*}
\sum_{\atopn{x\in\Gamma(c)\cap\widetilde{\Gamma}(d)}{| x| =J}}
{|<\mu[e,a,v_a],\pi(x)\mu[e,b,v_b]>|^2}=R(d)\,{\mathcal D}^{J-1}\, S(c),
\end{equation*}
where $R(d)$ is the row vector obtained as tensor product of the vector on the left
side of \eqref{mum} by its conjugate, i.e.
\begin{equation}\label{erre}
R(d)=
\left(\begin{array}{l}
((v_{b}\otimes \overline{v_{b}})\, \delta{(db)} \,\delta(d^{-1} d'))_{d'\in A}\\ \\
((E_{e d}\otimes \overline{v_{b}})\, \delta{(db)}\, \delta(d^{-1} d'))_{d'\in A}\\ \\
((v_{b}\otimes \overline{E_{e d}}) \,\delta{(db)}\, \delta(d^{-1} d'))_{d'\in A}\\ \\
((E_{e d}\otimes \overline{E_{e d}}) \,\delta(d^{-1} d'))_{d'\in A}
\end{array}\right)^{\top},
\end{equation}
and $S(c)$ is the column vector defined as tensor product of the vector on the
 right of \eqref{mum}, for $J=2,$ by its conjugate, i.e.
\begin{equation}\label{esse}
S(c)=
\left(\begin{array}{l}
((E_{c e}\otimes \overline{E_{c e}})\,\delta(c^{-1}c'))_{c'\in A}\\ \\
((v_{a}\otimes \overline{E_{c e}})\,\delta(c^{-1}a)\,\delta(c^{-1}c'))_{c'\in A}\\ \\
((E_{c e}\otimes \overline{v_{a}})\,\delta(c^{-1}a)\,\delta(c^{-1}c'))_{c'\in A}\\ \\
((v_{a}\otimes \overline{v_{a}})\,\delta(c^{-1}a)\,\delta(c^{-1}c'))_{c'\in A}
\end{array}\right).
\end{equation}
\end{thm}

The interested reader can find both proofs in Appendix.

\subsection{The  $1$-Eigenspace of $\mathcal{D}$}

%
%
The following calculations do not depend on having a free group.
We take an assigned indexing set
$A$, two systems
$(V_a, H_{ba})$ and $({V}^\sharp_a,{H}^\sharp_{ba}),$ (where
$H_{ba}:V_a\rightarrow {V_b},$ and ${H}^\sharp_{ba}:{V}^\sharp_a\rightarrow {V}^\sharp_b$) and  a set of linear maps
$E_{ba}:V_a\rightarrow {V_b}^{\sharp}.$
We shall denote  by
${\mathcal D}=(D_{i,j})_{i,j=1,\dots, 4}$
the following matrix
\begin{equation}\label{matrix}
{\mathcal D}=\!\!\left(\!\!\!\!\begin{array}{cccc}
\left(\!{H^\sharp}_{ab}\otimes \overline{{H^\sharp}}_{ab}\!\right)_{a,b}
&\!\!\left(\!E_{ab}\otimes \overline{{H^\sharp}}_{ab}\!\right)_{a,b}
&\!\!\!\left(\!{H^\sharp}_{ab}\otimes \overline{{E}}_{ab}\!\right)_{a,b}
&\!\!\!\left(\!{E}_{ab}\otimes \overline{{E}}_{ab}\!\right)_{a,b}\\ \\
\!\!0
&\!\!\!\left(\!{H}_{ab}\otimes \overline{{H^\sharp}}_{ab}\!\right)_{a,b}
&\!\!\!\!0
&\!\!\!\left(\!{H}_{ab}\otimes \overline{{E}}_{ab}\!\right)_{a,b}\\ \\
\!\!0
&\!\!\!\!0
&\!\!\!\left(\!{H^\sharp}_{ab}\otimes \overline{{H}}_{ab}\!\right)_{a,b}
&\!\!\!\left(\!{E}_{ab}\otimes \overline{{H}}_{ab}\!\right)_{a,b}\\ \\
\!\!0
&\!\!\!0
&\!\!\!\!0
&\!\!\!\left(\!{H}_{ab}\otimes \overline{{H}}_{ab}\!\right)_{a,b}
\end{array}
\!\!\right)
\end{equation}


\begin{prop}\label{pregressi}
\par
Assume that we are given two normalized irreducible systems $(H_{ba}, V_a)$,
$({V}^\sharp_a,{H^\sharp}_{ba})$ and let $\mathcal D$ be as in \eqref{matrix}.
Then the spectral radius of $\mcD$ is $1$. Moreover:
\begin{itemize}
\item[A)] If the two systems are  {\bf inequivalent},
$1$ is an eigenvalue of multiplicity {\bf two};
\item[B)] If the two systems are {\bf equivalent}, $1$ is an
 eigenvalue of multiplicity {\bf four}.
\end{itemize}

\end{prop}
\begin{proof}
Since $\mcD$ is block upper triangular, its eigenvalues are the same
as the eigenvalues of the diagonal blocks $D_{j,j}$.
By Lemma 4.6 of \cite{K-S3} and
Theorem 3.1 of \cite{V}
the two blocks $D_{1,1}$ and $D_{4,4}$ have spectral radius
equal to $1,$ and  $1$ is an eigenvalue irrespective
 of whether the two systems are  equivalent.
Let us turn to the other
diagonal  blocks.
Observing that the eigenvalues of the transpose matrix
${D'}_{j,j}$ are the same as those of $D_{j,j}$,
 one can apply Corollary 5.4 of \cite{K-S3}
to both matrices $D_{2,2}=({H}^{\sharp}_{ab}\otimes \overline{H}_{ab})$
and $D_{3,3}=
({{H}^{\sharp}}_{ab}^{\ast}\otimes{H}_{ab})$
to conclude that
\begin{enumerate}
\item
they both have spectral radius less or equal to $1$,
\item $1$ is an eigenvalue for both if and only if the two systems
are equivalent.
\end{enumerate}
Let us turn now to the multiplicity of $1$ for  $D_{1,1}$ and $D_{4,4}$.
We shall consider only $D_{1,1}$, {being} the other case similar.
By Corollary 4.8 of \cite{K-S3}, we know that there exists an essentially
unique tuple $U=(U_a)$ of strictly positive definite forms such that
$D_{1,1}U=U$, hence the geometric multiplicity of $1$ is $1$.
Assume, by contradiction, that the algebraic multiplicity is more than one.
Hence there exists a nonzero tuple $W=(W_a)$ satisfying
\begin{equation*}
D_{1,1}W=W+\lambda U\quad \mbox{for some $\lambda\neq0$}\;.
\end{equation*}
We may assume that $\lambda=1$
(if $\lambda$ is negative we may replace $W$ with $-W$).
Choose $t_0$ big enough so that $t_0U-W$ is positive semidefinite.

By our assumption $D^n_{1,1}(t_0U-W)$ is also positive
 semidefinite for all $n\geq0$. Since
$$
D_{1,1}^n(t_0U-W)
=t_0U-W-nU
$$
when $n$ is big enough we get a contradiction since $U$ is {\it strictly}
positive.
To conclude the proof observe that, when the two systems are equivalent, by
 Remark \ref{inequivalenti}, all
the diagonal blocks $D_{j,j}$ are similar to $D_{1,1}$.
\end{proof}

\subsection{Proof of Theorem 1}

\begin{reftheorem}{teorema1}
Let $(V_a,H_{ba})$
be an irreducible normalized matrix system.
Construct the matrix $\mcD$ as in \eqref{matrix} using for
 $(V^\sharp_a,H^\sharp_{ba})$  the twin  system
 $(\widehat{V}_a,\widehat{H}_{ba})$ and let $d$ be the dimension of
the eigenspace of $1$ of $\mcD$. For any
positive $\varepsilon$, $v_a\in V_a$ and $v_b\in V_b$ define
$$
\|\phi_\varepsilon^{v_a,v_b}\|^2=
\sum_{x\in\Gamma}
{|<\mu[e,a,v_a],\pi(x)\mu[e,b,v_b]>|^2}
e^{-\varepsilon|x|} \;.
$$
Then
\begin{itemize}
\item[A)] If the two systems are {\bf inequivalent}
 one has, as $\varepsilon\to0$
\begin{align*}
&\|\phi_\varepsilon^{v_a,v_b}\|^2\simeq\frac1\varepsilon \;&\text{when}\; d&=2, \\
&\|\phi_\varepsilon^{v_a,v_b}\|^2\simeq\frac1{\varepsilon^2}\;& \text{when}\; d&=1.
\end{align*}
\item[B)]
If the two systems are {\bf equivalent}
\begin{align*}
&\|\phi_\varepsilon^{v_a,v_b}\|^2\simeq\frac1\varepsilon\;&\text{when}\; d=4, \\
&\lim_{\eps\to0}\eps^3\|\phi_\varepsilon^{v_a,v_b}\|^2\;\text{exists and is finite}
& \text{in all other cases.}
\end{align*}


\end{itemize}
\end{reftheorem}
\begin{proof}
Cases A) and B) correspond exactly to items A and B of Proposition
\ref{pregressi}.
Define
\begin{equation}\label{matrixcoe}
\psi(\varepsilon,c,d)=
\sum_{J=1}^\infty \sum_{\atopn{x\in\Gamma(c)\cap\widetilde{\Gamma}(d)}{| x| =J}}
{|\!<\mu[e,a,v_a],\pi(x)\mu[e,b,v_b]>\!|^2}e^{-\varepsilon J}.
\end{equation}
It is enough to compute $\psi(\varepsilon, c,d)$
for all $c,d\in A$.
By Theorem \ref{mumu0} there exist vectors $R(d)$ and $S(c)$, depending only
on $v_a$ and $v_b$, such that
$$
\sum_{\atopn{x\in\Gamma(c)\cap\widetilde{\Gamma}(d)}{| x| =J}}
{|<\mu[e,a,v_a],\pi(x)\mu[e,b,v_b]>|^2}e^{-\varepsilon J}=
R(d)\,{\mathcal D}^{J-1}\, S(c)e^{-\varepsilon J}
$$
where $\mcD=\tilde{\mcD}\otimes\overline{\tilde{\mcD}}$
and $\tilde{\mcD}$ is as in \eqref{matrixtilde}. Observe that
$\mcD$ is the same as the matrix of equation \eqref{matrix}
where we set $H^\sharp_{ba}=\widehat{H}_{ba}$. Moreover, $\mcD$
depends only on the system we started with.
Denote by $\mathscr{L}$ the finite dimensional vector space on which
$\mcD$ acts and by $K_1$ the generalized eigenspace of $1$.
Since
$$
\psi(\varepsilon,c,d)
= e^{-\varepsilon}R(d)\,(I-{\mathcal D} e^{-\varepsilon})^{-1}\, S(c),
$$
the growth of $\psi(\varepsilon,c,d)$  as $\varepsilon$ goes to zero,
depends only on the maximum size of
the Jordan blocks $J_1$ relative to $K_1$. We recall that a Jordan
block of size $r$ will produce a leading term
$\sum_{J=1}^\infty\binom{J-1}{r-1}e^{-\eps J}$
in the computation of $\psi(\varepsilon,c,d)$.
When the two systems are inequivalent, by Proposition \ref{pregressi},
the dimension of $K_1$ is two
and
\begin{equation*}
\psi(\varepsilon,c,d)\simeq \sum_{J=1}^\infty e^{-\varepsilon J}\simeq
\frac1\varepsilon\quad\mbox{when }\; J_1=\left(\begin{matrix}
1&0\\
0&1
\end{matrix}\right)
\end{equation*}
or
\begin{equation*}
\psi(\varepsilon,c,d)\simeq \sum_{J=1}^\infty (J-1) e^{-\varepsilon J}\simeq
\frac1{\varepsilon^2}\quad\mbox{when }\; J_1=\left(\begin{matrix}
1&1\\
0&1
\end{matrix}\right)
\end{equation*}
More details can be found in the proof of Lemma \ref{lemma5.4}.

For equivalent systems there are two more possibilities for the maximum size of
Jordan blocks, namely three and four.
The existence of a single Jordan block of size four is ruled out by
Haagerup's condition \eqref{haag}, since a block of size four  would lead to
\begin{equation*}
\psi(\varepsilon,c,d)\simeq \sum_{J=1}^\infty \binom{J-1}{3}
 e^{-\varepsilon J}\simeq
\frac1{\varepsilon^4}\;.
\end{equation*}
The last possibility is a
Jordan block of size
three together with one of size one: this is exactly what
happens
 for the representations corresponding to the
endpoints of the  isotropic/anisotropic principal
series of Fig\`a-Talamanca and Picardello \cite{FT-P},
Fig\`a-Talamanca and Steger \cite{FT-S} for which one gets
\begin{equation*}
\psi(\varepsilon,c,d)\simeq \sum_{J=1}^\infty \binom{J-1}{2}
 e^{-\varepsilon J}\simeq
\frac1{\varepsilon^3}.
\end{equation*}
\end{proof}
\begin{rem}
A more accurate analysis of $\mcD$, that we shall omit here,
shows that, for equivalent systems,
it is not possible to have a Jordan block of length two, hence in this case
 there are only two possible behaviors
for $\|\phi_\varepsilon^{v_a,v_b}\|_2^2$, namely $\frac1{\varepsilon^3}$
or $\frac1\varepsilon$.
\end{rem}
\begin{lem}\label{lemq}
Let $(V_a, H_{ba})$, $({V}^{\sharp}_a,{H}^{\sharp}_{ba})$ and $\mcD$ be
as in Proposition~\ref{pregressi} and
let $(P_a)_a$ be the eigenvector of $1$ of  $D_{4,4}$.
Assume that the two systems are inequivalent.
Then for each $b\in A$, there exists a linear map
$Q_b:V_b\rightarrow {V}^{\sharp}_b$ such that
 the vector
$$\left(\begin{array}{c}
(P_a\, Q_a^{\ast})_{a}\\
( Q_a \, P_a)_{a}\\
(P_a)_{a}
\end{array} \right)$$
is (up to constant) the unique eigenvector of
$1$ of the the principal submatrix
\begin{eqnarray*}
\mathcal{D}_1&=&\!\!\left(\begin{array}{ccc}
\left({H}_{ab}\otimes \overline{{H}^{\sharp}}_{ab}\right)_{a,b}&
0&\left({H}_{ab}\otimes \overline{{E}}_{ab}\right)_{a,b}\\ \\
0&\left({H}^{\sharp}_{ab}\otimes \overline{{H}}_{ab}\right)_{a,b}&
\left({E}_{ab}\otimes \overline{{H}}_{ab}\right)_{a,b}\\ \\
0&0&\left({H}_{ab}\otimes \overline{{H}}_{ab}\right)_{a,b}
\end{array}
\right)\\  \\ \\
&=&\left(\begin{array}{ccc}
D_{2,2}&
0&D_{2,4}\\ \\
0&D_{3,3}&
D_{3,4}\\ \\
0&0&D_{4,4}
\end{array}
\right).
\end{eqnarray*}
obtained by deleting the rows and columns of $D_{1,1}$
\end{lem}
\begin{proof}
We first note that, since the system  $(V_a, H_{ba})$ is irreducible and normalized,
then  $P_a$ is strictly positive definite
as a form on $V_a^{\ast}$, and so self-adjoint when identified with
the map $P_a:V_a^{\ast}\rightarrow V_a.$
By Proposition \ref{pregressi}, $1$ is an eigenvalue of multiplicity one for
$\mcD_1$.
We look for a vector $R=(R_2,R_3,R_4)$
 satisfying $\mathcal{D}_1 R=R$ or equivalently
\begin{equation}\label{eq0}\left\{\begin{array}{l}
D_{4,4} R_4=R_4 \\
D_{3,3} R_3+ D_{3,4} R_4=R_3 \\
D_{2,2} R_2+ D_{2,4} R_4=R_2. \\
\end{array}\right.\end{equation}
If $R_4= 0$, then
$R_2=R_3=0$, so let us assume that $R_4\neq 0$.
The first equation in \eqref{eq0} yields that $R_4$ is proportional to
$P=(P_b)_{b}.$
Without loss of generality, we can assume that
$R_4=P$.

The second equation can be written as
$$(I-D_{3,3} ) R_3=D_{3,4} P,$$
where $I$ is the identity matrix and, given the nature of  the eigenvalues of
$D_{3,3}$, it has $R_3=(I-D_{3,3})^{-1}D_{3,4} P$ as a unique solution.

Now $P_b$ is strictly positive definite, so it is invertible as
 $P_b:V_b^{\ast}\rightarrow V_b,$ so the map
$$Q_b=R_{3,b}\,P_b^{-1}:V_b\rightarrow{V}^{\sharp}_b,\quad
R_3=(R_{3,b})_b,$$
is linear and $R_{3,b}=Q_b\,P_b.$

Hence, the second equation in the system rewrites, for any $a\in A$,
$$
\sum_{b\in A}{({H}^{\sharp}_{ab}\otimes \overline{H}_{ab})(Q_b\,P_b)+
(E_{ab}\otimes\overline{H}_{ab})(P_b)}=Q_a\,P_a,
$$
which is equivalent to
\begin{equation}\label{eqlemq}
\sum_{b\in A}{{H}^{\sharp}_{ab}Q_b\,P_b {H}_{ab}^{\ast}
+E_{ab}P_b{H}_{ab}^{\ast}}=Q_a\,P_a.
\end{equation}
Taking adjoint we get
\begin{eqnarray*}
& &\sum_{b\in A}{({H}_{ab}\otimes\overline{{H}^{\sharp}}_{ab})
(P_b\, Q_b^{\ast}) +(H_{ab}\otimes\overline{E}_{ab})(P_b)}\\
&=&\sum_{b\in A}{{H}_{ab}P_b\, Q_b^{\ast} {{H}^{\sharp}}_{ab}^{\ast}
+H_{ab}P_b{E}_{ab}^{\ast}}
=P_a Q_a^{\ast},
\end{eqnarray*}
and  the entry of $R_2$ corresponding to $a\in A$ is necessarily
$$R_{2,a}=P_a\, Q_a^{\ast}.$$
\end{proof}
\begin{thm}\label{dim2}
Let $(V_a, H_{ba})$,  $({V}^{\sharp}_a,{H}^{\sharp}_{ba})$ and
$P=(P_a)$ be as in Lem\-ma~\ref{lemq}.
Let $\widetilde{P^\sharp}=(\widetilde{{P}^{\sharp}}_a)$ be the
eigenvector
 of $1$ of $D_{1,1}'=(H^\sharp_{ab}\otimes\overline{H^\sharp}_{ab})'$.

Then the eigenspace of ${\mathcal D}$ corresponding to  eigenvalue $1$ has dimension $2$
if and only if
there exist linear maps $Q_b:V_b\rightarrow {V}^{\sharp}_b,$ $b\in A,$
satisfying \eqref{eqlemq} so that the quantity
\begin{eqnarray}
& &{E^\sharp}_0:= \label{e0}\\
& &\!\!\sum_{a,b\in A}{\mathrm{tr}}{
(\widetilde{{P}^{\sharp}}_a \, E_{ab}\, P_b\, Q_b^{\ast}\,{{H}^{\sharp}}_{ab}^{\ast}
+\widetilde{{P}^{\sharp}}_a\, {H}^{\sharp}_{ab}\, Q_b\, P_b\, {E}_{ab}^{\ast}
+\widetilde{{P}^{\sharp}}_a\, {E}_{ab}\, P_b\,{E}_{ab}^{\ast})},\nonumber
\end{eqnarray}
verifies ${E^\sharp}_0=0.$
\end{thm}
\begin{proof}
Let $({P^{\sharp}}_a)$ be the right eigenvector of the
{block}  $D_{1,1}$. It is obvious that the full matrix $\mcD$
has the vector
$U^\sharp=\left(\begin{array}{c}
{P^\sharp}_a\\
0\\
0\\
0\\
\end{array}\right)$ as right eigenvector of $1$. Hence
the eigenspace of $1$ of ${\mathcal D}$ has dimension $2$
 if and only if there exists an eigenvector $W=(W_j)_{j=1,\dots,4},$
not proportional to $U^\sharp$.
 $W$ is a right eigenvector if and only if the three last components
$(W_2,W_3,W_4)$ satisfy \eqref{eq0} and the full vector satisfies
\begin{equation}\label{eq1}
D_{1,1} W_1+D_{1,2} W_2+D_{1,3} W_3+D_{1,4} W_4=W_1.
\end{equation}

Since we are looking for a vector not proportional to $U^\sharp$
we may assume $W_4\neq 0$.

By Lemma \ref{lemq}
$$
\left(\begin{array}{c}
W_2\\ W_3\\ W_4\end{array}\right)\quad=\quad
\left(\begin{array}{c}
(P_a\, Q_a^{\ast})_{a}\\
( Q_a \, P_a)_{a}\\
(P_a)_{a}
\end{array} \right)$$ for suitable maps $Q_a:V_a\to V^\sharp_a$.

If we denote by $W_{1,a}$  the entry
of $W_1$ corresponding to $a,$ equation \eqref{eq1} can be rewritten as
\begin{equation}\label{eq1bis}
\begin{aligned}
&
W_{1,a}-\sum_{b\in A}{{H}^{\sharp}_{ab}W_{1,b}{{H}^{\sharp}}_{ab}^{\ast}}=\\
&\sum_{b\in A}{\left[ E_{ab}P_b Q_b^{\ast}{{H}^{\sharp}}_{ab}^{\ast}+
{{H}^{\sharp}}_{ab}Q_b P_b{E}_{ab}^{\ast}+
 E_{ab}P_b{E}_{ab}^{\ast}\right]}\\
 &:=T_a.
\end{aligned}
\end{equation}
Since the submatrix $D_{1,1}$ does have the eigenvalue $1$, equation
\eqref{eq1bis} will have a solution if and only if the vector ${T=}(T_a)$
belongs to  $\mathrm{Im}(I- D_{1,1})$, the image of $(I-D_{1,1})$.
But
$$\mathrm{Im}(I- D_{1,1})
=(\mathrm{Ker}((I- D_{1,1})'))^{\bot}$$
and $\mathrm{Ker}((I- D_{1,1})')$ is the one-dimensional subspace
 generated by ${\widetilde{P^\sharp}=(\widetilde{P^\sharp}_a)}$.
Hence the linear system \eqref{eq1} has a solution not proportional to $U^\sharp$
if and only if
\begin{eqnarray*}
0&=&\!\!\mathrm{tr}({\widetilde{P^\sharp}} T)\\
&=&\!\!\mathrm{tr}\left(\sum_{a\in A}{{\widetilde{P^\sharp}_a}
\sum_{b\in A}{\left[ E_{ab}P_b Q_b^{\ast}{{H}^{\sharp}}_{ab}^{\ast}+
{{H}^{\sharp}}_{ab}Q_b P_b{E}_{ab}^{\ast}+
 E_{ab}P_b{E}_{ab}^{\ast}\right]}}\right)\\
 &=&\!\!\sum_{a,b\in A}{\mathrm{tr}\left({\widetilde{P^\sharp}_a} E_{ab}
P_b Q_b^{\ast}{{H}^{\sharp}}_{ab}^{\ast}+
{{\widetilde{P^\sharp}_a}{H}^{\sharp}}_{ab}Q_b P_b{E}_{ab}^{\ast}+
 {\widetilde{P^\sharp}_a} E_{ab}P_b{E}_{ab}^{\ast}\right)}.
 \end{eqnarray*}
\end{proof}
%
We are interested in a more manageable form for ${E^\sharp}_0.$ This can be achieved by
an algebraic calculation.
\begin{prop}
 The quantity ${E^\sharp}_0$ defined in \eqref{e0} can be written as
 \begin{eqnarray}
 & &{E^\sharp}_0= \label{squa}\\
 & &\sum_{a,b\in A}{\mathrm{tr}\left({\widetilde{P^\sharp}_a} ({{H}^{\sharp}}_{ab}Q_b+
 E_{ab}-Q_a H_{ab})P_b({{H}^{\sharp}}_{ab}Q_b+ E_{ab}-Q_a H_{ab})^{\ast}\right)}.\nonumber
 \end{eqnarray}
 \end{prop}
 \begin{proof}
 The proof is straightforward after  multiplication of all terms in the right hand
  side of \eqref{squa}.

\end{proof}

We recall now a general result in linear algebra.
\begin{lem}\label{alge}
Let $A,B$ be two strictly positive definite matrices
and let $C$  be  a not necessarily square matrix.
Then $tr(ACBC^{\ast})\geq 0,$ and
$$tr(ACBC^{\ast})=0\Longrightarrow C=0.$$
\end{lem}

\begin{thm}\label{trace1}
The quantity ${E^\sharp}_0$ verifies ${E^\sharp}_0\geq 0,$ and
 $$
 {E^\sharp}_0=0
$$
if and only if
the linear maps $Q_b:V_b\to {V}^\sharp_b$ provided by Lemma \ref{lemq}
verify the conditions
$$
{{H}^{\sharp}}_{ab}Q_b+ E_{ab}=Q_a H_{ab},\; \forall {a,b\in A}.
 $$
 \end{thm}
\begin{proof}
Since $P_a$  and ${\widetilde{P^\sharp}_a}$ are strictly positive definite
for all $a,$ the result follows from the previous Lemma \ref{alge}.
\end{proof}
\begin{coro}\label{particular}
Assume that $A$ generates a free group and that\\ $(V_a, H_{ba},B_a)$ is
 an irreducible, normalized matrix system with inner products. Let
$(\widehat{V}_a,\widehat{H}_{ba},\widehat{B}_a)$ be the
twin system
as in Proposition \ref{der}. Assume that
the two systems are {\bf inequivalent.}
Then the matrix $\mcD$ defined in \eqref{matrix}
has two linearly independent eigenvectors of $1$ if and only if
there exist linear maps $Q_b:V_b\to\widehat{V}_b$ satisfying
\eqref{eqlemq} and
$$
{\widehat{H}}_{ab}Q_b+ E_{ab}=Q_a H_{ab},\; \forall {a,b\in A}.
 $$
\end{coro}
\begin{proof}
Construct the matrix $\mcD$ corresponding to the two
matrix systems $(V_a, H_{ba},B_a)$ and
$({V}^{\sharp}_a,{H}^{\sharp}_{ba},B_a^\sharp)=
(\widehat{V}_a,\widehat{H}_{ba},\widehat{B}_a)$.
Since the tuple $(B_a)$, respectively $(\widehat{B}_a) $,
 is the right eigenvector for the matrix
$(H^*_{ba}\otimes H'_{ba})_{a,b}$, respectively
$\left(\widehat{H}_{ba}^{\ast}\otimes {\widehat{H}}_{ba}'\right)_{a,b},$
a direct computation shows that
\begin{itemize}
\item
 The tuple $u=(B_{b^{-1}})_b$ is the right eigenvector of $1$ of the
 submatrix $D_{1,1}$ while
the tuple $\tilde{u}=(\widehat{B}_{a})_a$ is a right eigenvector of $1$
for the transpose matrix $(D_{1,1})'$.
\item
The tuple
$v=(\widehat{B}_{b^{-1}})_b$ is the  right eigenvector of $1$ of the submatrix
$D_{4,4}$
while the tuple $\tilde{v}=({B}_{a})_a$ is the right eigenvector of $1$ for
 the transpose matrix $(D_{4,4})'$.
\end{itemize}
 Apply now Theorem \ref{trace1}
 with
${P=v}=(\widehat{B}_{a^{-1}})_a$  and
${\widetilde{P^\sharp}=\widetilde{u}} =(\widehat{B}_{a})_a$.
\end{proof}
\begin{rem}\label{autovettori}
Observe that the matrix $\mcD$ has
\begin{equation}\label{autodx}
U=(U_j)_ {j=1,\dots, 4}=\left(\begin{array}{c}
({B}_{b^{-1}})_{b}\\
0\\
0\\
0\\
\end{array}
\right)
\end{equation}
as  right eigenvector of $1$, while the
dual matrix $\mcD'=((D_{j,i})')$
has the vector
\begin{equation}\label{autost}
\tilde{V}=(\tilde{V}_j)_ {j=1,\dots, 4}=\left(\begin{array}{c}
0\\
0\\
0\\
({B}_{a})_{a}
\end{array}\right)
\end{equation}
 as right eigenvector of $1$.
\end{rem}

The following result is essential for the
computation of \eqref{matrixcoe}.

\begin{prop}\label{wdoppio}
Let $(V_a, H_{ba},B_a)$ be
 an irreducible, normalized matrix system and let $(\widehat{V}_a,\widehat{H}_{ba}, \widehat{B}_a)$
be the twin system. Assume that
the  two systems  are inequivalent.
Define $Q_b: V_b\to\widehat{V_b}$ as in Lemma \ref{lemq} and let
 \begin{eqnarray}\label{e00}
 & &{E}_0= \\
 & &\sum_{a,b\in A}{\mathrm{tr}\!\left(\widehat{B}_a ({\widehat H}_{ab}Q_b+
 E_{ab}-Q_a H_{ab})\widehat{B}_{b^{-1}}(\widehat{ H}_{ab}Q_b+ E_{ab}-Q_a H_{ab})^{\ast}\right)}\;.\nonumber
 \end{eqnarray}
If $E_0\neq0$ there exists
a vector $W\neq 0$
such that
\begin{equation}\label{autogen}
{\mathcal D}\,W=W+\lambda U,
\end{equation}
where $U$  is, as in \eqref{autodx}
the right eigenvector of $1$ of  ${\mathcal D}$ and
$$
\lambda=
\frac{E_0}{\mathrm{tr}\left(\sum_{a\in A}{\widehat{B}_a B_{a^{-1}}}\right)}.
$$
\end{prop}
\begin{proof}
By Proposition \ref{pregressi} and Corollary \ref{particular}
the Jordan block of  ${\mathcal D}$ corresponding to eigenvalue $1$
is of the form $\left(\begin{array}{ll}
1&1\\
0&1
\end{array}\right)$.

We are seeking for a nonzero vector $W$ such that
${\mathcal D}\,W=W+\lambda U$ for some nonzero $\lambda$.

 Write
\begin{align}
W &=(W_i)_{i=1,\dots,4}=((W_{i,a})_{a})_{i=1,\dots,4},\nonumber\\
U &=(U_j)_{j=1,\dots,4}=(\delta_{j1}(B_{b^{-1}})_{b})_ {j=1,\dots, 4},\quad
\delta_{j,i}=\left\{\begin{aligned}
& 1\quad\mbox{if }\; i=j\\
& 0\quad\mbox{if }\; i\neq j,
\end{aligned}\right.
\nonumber
\end{align}
and use
Lemma \ref{lemq} (with $V^\sharp_a=\widehat{V}_a$, $H^\sharp_{ab}=\widehat{H}_{ab}$,
${P_a}=(\widehat{B}_{a^{-1}})$)   to get
the three last components of $W$:
\begin{equation*}
\left(\begin{array}{c}
W_2\\ W_3\\ W_4\end{array}\right)\quad=\quad
\left(\begin{array}{c}
(\widehat{B}_{b^{-1}}\, Q_b^{\ast})\\
(Q_b\,\widehat{B}_{b^{-1}})\\
(\widehat{B}_{b^{-1}})\end{array}
\right)
\end{equation*}

Let us turn to the condition about the first component $W_{1}$.
Write, as in Theorem \ref{dim2}:
\begin{eqnarray*}
& &(I- D_{1,1})W_1=\\
& &\!\!\!\!\!
\left(\!\!-\lambda B_{a^{-1}}+\sum_{b\in A}{\left[ E_{ab}\widehat{B}_{b^{-1}} Q_b^{\ast}{\widehat{H}}_{ab}^{\ast}+
{\widehat{H}}_{ab}Q_b \widehat{B}_{b^{-1}}{E}_{ab}^{\ast}+
 E_{ab}\widehat{B}_{b^{-1}}{E}_{ab}^{\ast}\right]}\right)_{a}\!\!,
\end{eqnarray*}
and require that   the right hand side is perpendicular to the kernel
of
${(I- {D_{1,1}})'}$, which
is the one  dimensional subspace generated by
 $\tilde{u}=(\widehat{B}_{a})_a$.
As in Theorem \ref{dim2} this means to require
$$0=-\lambda\;\mathrm{tr}\left(\sum_{a\in A}{\widehat{B}_a B_{a^{-1}}}\right)+ E_0,$$
i.e.
$$\lambda=
\frac{ E_0}{\mathrm{tr}\left(\sum_{a\in A}{\widehat{B}_a B_{a^{-1}}}\right)}\neq 0.$$
\end{proof}
%

 \begin{lem}\label{lemma5.4}
Let $(V_a, H_{ba},B_a)$ be a normalized irreducible system and let
 $(\widehat{V}_a,\widehat{H}_{ba}, \widehat{B}_a)$ be
the twin system. Assume that the two systems are inequivalent.
For any $\varepsilon>0$, $v_a\in V_a$ and $v_b\in V_b$ let
\begin{equation*}
\|\phi_\varepsilon^{v_a,v_b}\|^2_2=
\sum_{x\in\Gamma}{|<\mu[e,a,v_a],\pi(x)\mu[e,b,v_b]>|^2 e^{-\varepsilon |x|}}\;.
\end{equation*}
Then $\|\phi_\varepsilon^{v_a,v_b}\|^2_2\simeq1/\varepsilon$ if and only if
the quantity $E_0$, defined in \eqref{e00}, is equal to zero.
Moreover, when
$E_0\neq0$ one has
$$
\|\phi_\varepsilon^{v_a,v_b}\|^2_2=
 \frac{E_0}{k_0^2}\varepsilon^{-2}B_a(v_a,v_a)B_b(v_b,v_b)
 + o(\varepsilon^{-2}),\quad \mathit{as}\quad \varepsilon\rightarrow 0,
$$
where
 $k_0=\sum_{c\in A}{\mathrm{tr}(\widehat{B}_c B_{c^{-1}})}.$
 \end{lem}
\begin{proof}
The first assertion follows from Theorems \ref{teorema1} and \ref{dim2}.
Assume hence that $E_0\neq0$ and use Theorem \ref{mumu0} to compute
\begin{eqnarray}
& &\sum_{J=1}^{+\infty}
{\sum_{\atopn{|x|=J}{x\in\Gamma(c)\cap\widetilde{\Gamma}(d)}}
{|<\mu[e,a,v_a],\pi(x)\mu[e,b,v_b]>|^2 e^{-\varepsilon|x|}}}\label{mumuep}\\
& &=
e^{-\varepsilon}R(d)\,(I-{\mathcal D} e^{-\varepsilon})^{-1}\, S(c),\nonumber
\end{eqnarray}
where vectors $R(d)$ and $S(c)$ are defined in
\eqref{erre}, \eqref{esse} and depend only on $v_a$ and $v_b$.

 Let us estimate the quantity on the right side in the above equality.
Denote by $\mathscr{L}$
the finite dimensional space on which $\mcD$
acts and by $K_1$ the generalized eigenspace of $1$.
Use Corollary~\ref{particular} and Proposition~\ref{wdoppio} to see that
$K_1$ is spanned by the vectors $U$ and $W$
provided by equations \eqref{autodx} and \eqref{autogen} and
take a basis of $\mathscr{L}$
which starts with $U$,  $W$ and
 ends with
generalized eigenvectors of ${\mathcal D}$
corresponding to eigenvalues
 different from $1$.
%
With respect to this basis $\mcD$ has the following
expression:
$$
\mcD=\left(
 \begin{array}{ccccc}
 1&\lambda&0&\dots& 0\\
 0&1&0&\dots& 0\\
 \vdots&\vdots&  & \mcF & \\
 0&0& &\dots&
 \end{array}\right)
$$
where the matrix $\mcF$ does not have the eigenvalue $1$.
Then
 $$
 (I-{\mathcal D} e^{-\varepsilon})^{-1}=\left(
 \begin{array}{ccccc}
 0&\lambda\varepsilon^{-2}&0&\dots& 0\\
 0&0&0&\dots& 0\\
 \vdots&\vdots& \vdots&\ddots &\vdots\\
 0&0&0&\dots& 0
 \end{array}\right)+O(\varepsilon^{-1}),\quad\mathrm{as}\quad\varepsilon\rightarrow 0.$$
For every vector $S\in\mathscr{L}$ write
 $S=s_1\,U+s_2\,W+\,u_{\mathrm{other}},$ where
$u_{\mathrm{other}}$ has a zero component in  $K_1$.
Then
$$(I-{\mathcal D} e^{-\varepsilon})^{-1}S=
\lambda\varepsilon^{-2}\,s_2 \,U+O(\varepsilon^{-1})(U+W+u_{\mathrm{other}}).$$
So, for any (row) vector $R\in\mathscr{L},$
\begin{eqnarray*}
e^{-\varepsilon}R\,(I-{\mathcal D} e^{-\varepsilon})^{-1}\, S&=&
e^{-\varepsilon}\lambda\varepsilon^{-2}\,s_2 R\,U+o(\varepsilon^{-2})   \\
&=&\lambda\varepsilon^{-2}\,s_2 R\,U+o(\varepsilon^{-2}).
\end{eqnarray*}

Let us denote by  $S_2$ the linear functional on $\mathscr{L}$
which associates the
second coordinate $s_2$ in our chosen basis.
According to   \eqref{otto}, one has
 $$S_2(S)={\textrm{tr}}(S_2\, S)$$
for a suitable row vector 
that we still denote by $S_2$.

We claim that $S_2$ is a left eigenvector of ${\mathcal D}$ corresponding to
eigenvalue $1$. Indeed, for
 $S=s_1\,U+s_2\,W+u_{\mathrm{other}},$ we have
\begin{eqnarray*}
S_2(I-{\mathcal D})S&=&S_2(I-{\mathcal D})
(s_1\,U+s_2\,W+ u_{\mathrm{other}})\\
&=&S_2(s_2\,W-s_2\,W-s_2\lambda\, U+(I-\mcD)u_{\mathrm{other}})\\
&=&S_2(-s_2\lambda\, U+0\cdot W+
w_{\mathrm{other}})\\
&=&0.
\end{eqnarray*}
As observed in Remark \ref{autovettori},
$S_2$ is proportional to the transpose vector of
 $\tilde{V}$ as defined in \eqref{autost}, so that there exists
 $\beta\in\C$ such that
$$S_2=(\begin{array}{cccc}
    0&0&0&\,\beta({B}_{a})_{a\in A}
\end{array}).$$
To compute $\beta$ let us recall that
$$W=\left(\begin{array}{c}
\ast\\
\vdots\\
\ast\\
(\widehat{B}_{a^{-1}})_{a\in A}
\end{array}\right),$$
hence
$$1=S_2(0\cdot u+W+0\cdot u_{\mathrm{other}})= {\textrm{tr}}(S_2\,W)=
\beta\sum_{a\in A}{{\textrm{tr}}({B}_{a}\,\widehat{B}_{a^{-1}})},$$
yields
$$\beta=\frac{1}{\sum_{a\in A}{{\textrm{tr}}({B}_{a}\,\widehat{B}_{a^{-1}})}}.$$
Finally, specifying  $R=R(d)$ and $S=S(c)$ defined  in \eqref{erre}
and\eqref{esse},
 \begin{eqnarray*}
R(d)\,(I-{\mathcal D} e^{-\varepsilon})^{-1}\, S(c)&=&
\varepsilon^{-2}\,\frac{E_0}{k_0^2}\,
{\textrm{tr}}(S_2 S(c))( R(d)\,U)+o(\varepsilon^{-2}).\end{eqnarray*}
The trace is given by
$${\textrm{tr}}(S_2 S(c))=
\sum_{s\in A}{{\textrm{tr}}(
{B}_{s}(v_{a}\otimes \overline{v_{a}})\,\delta(c^{-1}a)\,\delta(c^{-1}s))}=
B_c(v_c,v_c)\delta(c^{-1}a),$$
while
$$R(d)U=\sum_{r\in A}{(v_{b}\otimes \overline{v_{b}})\,
\delta{(db)} \,\delta(d^{-1} r)\,B_{r^{-1}}}=
B_{d^{-1}}(v_{d^{-1}},v_{d^{-1}})\delta{(db)}.$$
By summation on $c,d\in A$ we get from  \eqref{mumuep}
\begin{eqnarray*}
& &\sum_{x\in\Gamma}{|<\mu[e,a,v_a],\pi(x)\mu[e,b,v_b]>|^2 e^{-\varepsilon |x|}}\\
&=&
 \varepsilon^{-2}\,\frac{E_0}{k_0^2}
 \sum_{c,d\in A}{B_c(v_c,v_c)\delta(c^{-1}a)\,
  B_{d^{-1}}(v_{d^{-1}},v_{d^{-1}})\delta{(db)}}+o(\varepsilon^{-2})\\
 &=&
 \varepsilon^{-2}\,\frac{E_0}{k_0^2}\,B_a(v_a,v_a)\,
  B_{b}(v_{b},v_{b})+o(\varepsilon^{-2}).
  \end{eqnarray*}
\end{proof}

We proceed now with the computation of the limits that are needed to prove
Theorem \ref{main}.

\begin{lem}\label{lem3.2}
Let $c\in A,$ $f_1,f_2\in \mathcal{H}$ such that
$\mathrm{supp\,}f_i\subset\Gamma\setminus\Gamma(c),$ and $g_1,g_2\in \mathcal{H}.$
Then
\begin{eqnarray*}
& &\limsup_{\varepsilon\rightarrow 0^+}
\varepsilon\sum_{x\in\Gamma(c)}
{|<f_1,\pi(x)g_1><f_2,\pi(x)g_2>|\; e^{-\varepsilon |x|}}\\
& &\leq
\|f_1\|\|f_2\|\|g_1\|\|g_2\|\;.
\end{eqnarray*}
\end{lem}
\begin{proof}
See Lemma 3.2 of \cite{K-S2}.
\end{proof}
\begin{lem}\label{sfdef}
Let $c\in A$ be fixed and let $f\in {\mathcal H}$ such that
$\mathrm{supp\,}f\subset\Gamma\setminus\Gamma(c).$
If $g\in{\mathcal{H}}^{\infty},$ $x\in \widetilde{\Gamma}(c^{-1}),$ is of suitable length, then
\begin{equation}\label{preliminari}
<f,\mu[c,e,g(x)]>=<f,\pi(x^{-1})g>.
\end{equation}
\end{lem}
\begin{proof} By the previous Lemma
we may approximate $f$  with functions in ${\mathcal H}^{\infty}$
supported in  $\Gamma~\setminus~\Gamma(c)$.
For those functions one has
$$
<f,\pi(x^{-1})g> =
\sum_{|z|=N}\sum_{\atopn{a\in A }{|za|=N+1}}{B_a(f(za),g(xza))},
$$
where $N$ is big enough so  that both $f$ and $\pi(\inv x)g$
are multiplicative for $|z|>N$.

Since $f$ vanishes on words starting with $c$ and $x\in\widetilde\Gamma(\inv c)$
all the words $xza$ appearing in the above sum are reduced. Moreover, since
$g$ is multiplicative, one has
$$
g(xza)=\mu[c,e,g(x)](za)
$$
and  \eqref{preliminari} follows by adding up over all $z$.
\end{proof}

\begin{lem}\label{lem3.3}
Let $c\in A,$  $f\in \mathcal{H}$ such that
$\mathrm{supp\,}f\subset\Gamma\setminus\Gamma(c);$ let $g\in\mathcal{H}.$ Then
there exists an absolute constant, $k_0>0$ and there exists the limit
$$\lim_{\varepsilon\rightarrow 0^+}{\varepsilon\sum_{x\in\Gamma(c)}
{|<f,\pi(x)g> |^2 e^{-\varepsilon| x|}}}=
\frac{1}{k_0}\,\widehat{B}_c(Sf,Sf)\,\|g\|^2,$$
where $Sf$ is the antilinear functional on $V_{c^{-1}}$
defined by the rule
$$
Sf({v_{c^{-1}}})=<f,\mu[c,e,v_{c^{-1}}]>
$$
and $k_0=\sum_{a\in A}{\mathrm{tr}(\widehat{B}_a B_{a^{-1}})}$.
\end{lem}
\begin{proof}
By Lemma \ref{lem3.2} and density we can reduce to the case
$g\in\mathcal{H}^{\infty}.$
Identify $Sf$ with an element of $\widehat{V}_c=\overline{V}_{c^{-1}}'$ and
$Sf\otimes\overline {Sf}$ with an element
of ${\mathscr{L}({V}_{c^{-1}},\overline{V}_{c^{-1}}')}$ =
${\mathscr{L}({V}_{c^{-1}},{V}_{c^{-1}}^{\ast})}$=
${\mathscr{L}({V}_{c^{-1}}^{\ast},{V}_{c^{-1}})'}$.
By  Lemma \ref{sfdef},
if $x^{-1}\in\Gamma(c)$ has length $N$ big enough,
\begin{equation*}
Sf(g(x))=<f,\mu[c,e,g(x)]>=<f,\pi(\inv x)g>\;.
\end{equation*}
Identify  $g(x)\otimes\overline{g(x)}$ with an element
of
${\mathscr{L}(\bar V_{c^{-1}}',V_{c^{-1}})}=
{\mathscr{L}({V}_{c^{-1}}^{\ast},{V}_{c^{-1}})}$
and recall the duality expressed in
\eqref{nove} to get
$$
(Sf\otimes \overline{Sf}) (g(x)\otimes\overline{g(x)})=
|<f,\pi(x^{-1})g>|^2,$$
(here we proceed with $x^{-1}$ instead of $x$ by sake of calculation).

For the purpose of the limit the contribution of $x$ such that $|x|<N$
 is irrelevant, hence it in enough to compute
\begin{eqnarray*}
& &\sum_{\atopn{x^{-1}\in{\Gamma}(c)}{|x|\geq N}}{|<f,\pi(x^{-1})g>|^2 \, e^{-\varepsilon| x|}}\\
&=&
\sum_{\atopn{x\in\widetilde{\Gamma}(c^{-1})}{|x|\geq N}}
{(Sf\otimes \overline{Sf})( g(x)\otimes\overline{g(x)}) \, e^{-\varepsilon| x|}}\;.
\end{eqnarray*}
Since the trace is linear and continuous, we shall focus on
\begin{equation}\label{sfg}
e^{-\varepsilon N}
{({Sf}\otimes\overline{Sf}) \sum_{n=0}^{+\infty} \left(
\sum_{\atopn{x\in\widetilde{\Gamma}(c^{-1}),\, x=y\cdot c^{-1}}{|x|=|y|+1=n+ N}}
{g(x)\otimes\overline{g(x)}}\right) \, e^{-\varepsilon n}}.
\end{equation}
Now we set, for any $b\in A,$ and $n\in\N,$
$$\beta_{n+N,b}=\sum_{\atopn{x\in\widetilde{\Gamma}(b),\, x=y\cdot b}{|x|=|y|+1=n+ N}}
{g(x)\otimes\overline{g(x)}}\in \mathscr{L}({V}_{b}^{\ast},{V}_{b}),$$
which defines the (column) vector $\beta_{n+N}=(\beta_{n+N,b})_{b},$
and
\begin{equation}\label{sfsf}
F_b=({Sf}\otimes \overline{Sf})\,\delta(bc)
\in {\mathscr{L}({V}_{b}^{\ast},{V}_{b})'},
\end{equation}
which defines  the (row) vector
$F=(F_b)_{b}.$

Recall from \eqref{matrix} the matrix $D_{4,4}=\left({H}_{ab}\otimes \overline{{H}}_{ab}\right)_{a,b}$.
We show first that $D_{4,4}\beta_N=\beta_{1+N}.$
Indeed, since $g$ is a multiplicative function, for any $a\in A,$
\begin{eqnarray*}
D_{4,4}\beta_N
&=&\sum_{b\in A}\sum_{\atopn{x\in\widetilde{\Gamma}(b),\, x=y\cdot b}{|x|=|y|+1= N}}
{H_{ab}g(x)\otimes\overline{H_{ab}g(x)}}\\
&= &\sum_{\atopn{b\in A }{b\neq a^{-1}}}
\sum_{\atopn{x=y\cdot b\in\widetilde{\Gamma}(b)}{|x|=|y|+1=N}}
{g(yba)\otimes\overline{g(yba)}}
=
\sum_{\atopn{z=y\cdot a\in\widetilde{\Gamma}(a)}{|z|=|y|+1= N+1}}
{g(z)\otimes\overline{g(z)}}.
\end{eqnarray*}
And, by iteration, for any $n$ we get $(D_{4,4})^n\beta_N=\beta_{n+N}.$

From \eqref{sfg} and \eqref{sfsf} we can write
\begin{eqnarray*}
{(Sf\otimes \overline{Sf})
\sum_{\atopn{x\in\widetilde{\Gamma}(c^{-1})}{|x|\geq N}}
( g(x)\otimes\overline{g(x)}) \, e^{-\varepsilon| x|}}
& =&\!\!e^{-\varepsilon N}  F \left(\sum_{n=0}^{+\infty}
{(D_{4,4}\, e^{-\varepsilon})^n\beta_N }\right),
\end{eqnarray*}
where the hypotheses on the matrix systems guarantee that the series converges.

The limit that we are interested in is therefore
$$
\lim_{\varepsilon\rightarrow 0^+}{\varepsilon\sum_{x^{-1}\in{\Gamma}(c)}
{|<f,\pi(x^{-1})g>|^2 \, e^{-\varepsilon| x|}}}
=  F \left[\lim_{\varepsilon\rightarrow 0^+}{\varepsilon\,\sum_{n=0}^{+\infty}
{(D_{4,4}\, e^{-\varepsilon})^n\beta_N }}\!\right],$$
the calculation of which we provide in the following claim.

{\bf Claim} Let $D_{\varepsilon}=D_{4,4}\, e^{-\varepsilon}$. Then
\begin{eqnarray}\label{fv}
& &  F \left[\lim_{\varepsilon\rightarrow 0^+}{\varepsilon\,\sum_{n=0}^{+\infty}
{D_{\varepsilon}^n\beta_N }}\right]=
\frac{\widehat{B}_c (Sf,Sf)}{\sum_{a\in A}{\mathrm{tr}(\widehat{B}_a B_{a^{-1}})}}
\|g\|^2.
\end{eqnarray}
{\bf Proof of the Claim}\\
The quantity in bracket
is a right eigenvector for the matrix $D_{4,4}$, corresponding to  eigenvalue $1.$
Indeed, since $D_{4,4}=\lim_{\varepsilon\rightarrow 0^+}D_{\varepsilon},$
$$
D_{4,4}[\lim_{\varepsilon\rightarrow 0^+}
{\varepsilon  \sum_{n=0}^{+\infty}{D_{\varepsilon}^n \beta_N}}]=
\lim_{\varepsilon\rightarrow 0^+}
{\varepsilon  \sum_{n=0}^{+\infty}{D_{\varepsilon}^{n+1} \beta_N}}
=\lim_{\varepsilon\rightarrow 0^+}
{\varepsilon \sum_{n=0}^{+\infty}{D_{\varepsilon}^{n} \beta_N}}.
$$

But, up to constants, $v=(\widehat{B}_{b^{-1}})_{b}$ is the only right
eigenvector of $D_{4,4}$ corresponding to eigenvalue $1$. Hence the two vectors must
be proportional, and there exists
$\alpha\in\C$ such that the left hand side of \eqref{fv} is equal to
$$
  F\alpha\, v=\alpha   (Sf\otimes \overline{Sf})( \widehat{B}_c)=
\alpha\overline{Sf}( \widehat{B}_c (Sf))
=\alpha \overline{\widehat{B}_c(Sf,Sf)}=
\alpha \widehat{B}_c(Sf,{Sf}).
$$

Let us calculate $\alpha$.
As follows from the proof of Corollary \ref{particular},
the transpose vector $\tilde{v}^{\top}$ of
$\tilde{v}=({B}_{b})_{b}$ is a left eigenvector
of $D_{4,4}$ corresponding to  eigenvalue $1$.
Hence
\begin{eqnarray*}
\alpha\sum_{b\in A}\mathrm{tr}(B_b \widehat{B}_{b^{-1}} )&=&
\tilde{v}^{\top}\alpha v=
\lim_{\varepsilon\rightarrow 0^+}
{\varepsilon \sum_{n=0}^{+\infty}{\tilde{v}^{\top} D_{\varepsilon}^{n} \beta_N}}\\
&=&
\lim_{\varepsilon\rightarrow 0^+}
{\varepsilon \sum_{n=0}^{+\infty}{e^{-n\varepsilon}\tilde{v}^{\top}\beta_N}}
=\tilde{v}^{\top} \beta_N
= \sum_{b\in A}\mathrm{tr}({B}_b \,\beta_{N,b}),
\end{eqnarray*}
and we obtain
$$\alpha
=\frac{\sum_{b\in A}\mathrm{tr}({B}_b \,\beta_{N,b})}
{\sum_{b\in A}\mathrm{tr}(B_b \widehat{B}_{b^{-1}} )}.$$
Finally
\begin{eqnarray*}
\sum_{b\in A}\mathrm{tr}({B}_b \,\beta_{N,b})&=&
\sum_{b\in A}\mathrm{tr}(\beta_{N,b}\,{B}_b)
=
\sum_{b\in A}\sum_{\atopn{x=y\cdot b\in\widetilde{\Gamma}(b)}{|x|=|y|+1= N}}
B_b(g(x),g(x))\\
&=&\sum_{|y|=N-1}\sum_{\atopn{b\in A}{|yb|=|y|+1}}{\!B_b(g(yb),g(yb))}
=<g,g>=\|g \|^2.
\end{eqnarray*}
This ends both the proof of the Claim and of the Lemma.
\end{proof}
\begin{coro}\label{coro3.4}
Let $y=zc\in\widetilde{\Gamma}(c)$. Let $f\in \mathcal{H}$ be such that
$\mathrm{supp}f\subset\Gamma\setminus\Gamma(y).$ Let $g\in\mathcal{H}.$
Then
$$\lim_{\varepsilon\rightarrow 0^+}{\varepsilon\sum_{x\in\Gamma(y)}
{|<f,\pi(x)g>|^2 e^{-\varepsilon| x|}}}=
\frac{1}{k_0}\,\widehat{B}_c(S\pi(z^{-1})f,S\pi(z^{-1})f)\, \|g\|^2,$$
where the constant $k_0$ is given in the previous Lemma.
\end{coro}
\begin{proof}
If $x\in\Gamma(y)$ and $x=zct$ then $\pi(x)=\pi(z)\pi(ct).$
It is sufficient to apply the previous Lemma to $f_1=\pi(z^{-1})f$ which has
support in $\Gamma\setminus\Gamma(c).$
\end{proof}

\section{Appendix}

In order to prove Lemma \ref{mum0} and Theorem \ref{mumu0} we need some notation.
For any $x=a_1 a_2\dots a_J\in \Gamma,$ $a_i\in A,$  reduced word, we define
$$\begin{array}{ll}
x_0=e,&\\
x_j=a_1 a_2\dots a_j,& \text{for}\;j=1,\dots,J,\\
\tilde{x}_j=a_{j+1} a_{j+2}\dots a_J,& \text{for}\;j=0,\dots,J-1,\\
\tilde{x}_J=e,&\\
x_J=\tilde{x}_0=x.&\\
\end{array}$$
If  $x=a_1 a_2\dots a_J$ is fixed, we introduce $J+1$ subsets of $\Gamma$, namely
if $j=1,\dots,J-1,$
$$\begin{array}{l}
\Lambda_0=\{z\in\Gamma,\, | z| =J,\, z=b_1 b_2\dots b_J,\, b_1\neq a_1\},\\
\Lambda_j=\{z\in\Gamma,\, | z| =J,\, z=a_1 \dots a_j b_{j+1}\dots b_J,
\, b_{j+1}\neq a_{j+1}\}, \\
\Lambda_J=\{x\}.
\end{array}$$
So that,  for $j=1,\dots,J-1$ and $z\in \Lambda_j$
$$x^{-1}z=a_J^{-1}a_{J-1}^{-1}\dots a_2^{-1}a_1^{-1}b_1 b_2\dots b_J=
a_J^{-1}\dots a_{j+1}^{-1}b_{j+1} \dots b_J=\tilde{x}_j^{-1}\tilde{z}_j,$$
hence
$$\Lambda_j=\{z\in\Gamma,\, | z| =| z_j| +|\tilde{z}_j|,\, x_j=z_j,\,
|\tilde{x}_j^{-1}\tilde{z}_j|=|\tilde{x}_j^{-1}|+|\tilde{z}_j|\}.$$
Furthermore for a word $y=b_1\dots b_n,$ $b_i\in A,$ we write in short
\begin{equation}\label{xj}
\begin{array}{l}
H(y)=H_{b_n b_{n-1}}\dots H_{b_2 b_{1}},\quad H(b_1)=\textrm{id},\\
\hat{H}(y)=H(y^{-1})^{\ast}=\hat{H}_{b_n b_{n-1}}\dots \hat{H}_{b_2 b_{1}}.
\end{array}
\end{equation}

We recall that the positive definite sesquilinear form $B_a$
on the space $V_a$ is identified with
a map $B_a \in\mathscr{L}(V_a,V_a^{\ast}),$ hence we shall often use the equality
$B_a(H_{ab}v,H_{ab}w)=H_{ab}^{\ast}B_a H_{ab}(v,w),$
for any $v,w\in V_b.$

{\bf Proof of Lemma \ref{mum0}}

Let $J\geq 1,$ $x=a_1 a_2\dots a_J$ be a reduced word,
$a,b\in A,$ $v_a\in V_a$, $v_b\in V_b$.

By definition of the inner product in $\mathcal{H}$ and the representation $\pi$,
since the set of words of length $J$ is the disjoint union of the sets $\Lambda_j$,
$j=0,\dots, J,$ we have
\begin{eqnarray*}
& &<\mu[e,a,v_a],\pi(x)\mu[e,b,v_b]>\\  \\
&=&\sum_{| z|=J}\sum_{\atopn{c\in A }{| z c|=J+1}}{B_c(\mu[e,a,v_a](zc),\mu[e,b,v_b](x^{-1}z c))}\\
&=&\sum_{j=0}^J\sum_{z\in \Lambda_j}\sum_{\atopn{c\in A }{| z c|=J+1}}
{B_c(\mu[e,a,v_a](x_j\tilde{z}_j c),\mu[e,b,v_b](\tilde{x}_j^{-1}\tilde{z}_j c))}
\end{eqnarray*}
\begin{eqnarray*}
&=&\sum_{z\in \Lambda_0}\sum_{\atopn{c\in A }{| z c|=J+1}}{\dots}
+\sum_{j=1}^{J-1}\sum_{z\in \Lambda_j}
\sum_{\atopn{c\in A }{| z c|=J+1}}
{\dots}+\sum_{z\in \Lambda_J}\sum_{\atopn{c\in A }{| z c|=J+1}}{\dots}
\end{eqnarray*}

We write  both
$\mu[e,a,v_a](x_j\tilde{z}_j c)$ and  $\mu[e,b,v_b](\tilde{x}_j^{-1}\tilde{z}_j c)$
 as products of different $H\!$s in the three  sums above.
   We obtain that the latter is equal to
\begin{eqnarray*}
& & \sum_{\atopn{|z|=J}{z=b_1\dots b_J}}
\left[\sum_{\atopn{c\in A }{c b_J \neq e}}
{H_{c b_J}^{\ast}B_c H_{c b_J}}\right](H(z)(v_{a}),
H(z)  H_{b_1 a_1^{-1}} H(x^{-1})(v_{b}))\cdot\\
& &\hspace{8,5cm}\cdot\delta{(b_1 a^{-1})}\delta{(a_J b)}\\
&  &
+\sum_{j=1}^{J-1}\sum_{\atopn{z\in \Lambda_j}
{z=a_1\dots a_jb_{j+1}\dots b_J}}
\left[\sum_{\atopn{c\in A }{| z c|=J+1}}
{H_{c b_J}^{\ast}B_c H_{c b_J}}\right](H(\tilde{z}_j)H_{ b_{j+1}a_j} H(x_j)(v_{a}),\\ \\
& &\hspace{4cm} H(\tilde{z}_j)H_{ b_{j+1}a_{j+1}^{-1}}
 H(\tilde{x}_j^{-1})(v_{b}))\,\cdot\delta{(a_1 a^{-1})}\delta{(a_J b)}\\
& &+
{B_b(H_{c a_J} H(x)(v_{a}),v_b)\,\cdot\delta{(a_1 a^{-1})}}.
\end{eqnarray*}
We can apply repeatedly  the identity \eqref{compatibility} which characterizes
 $B_c$ in the first sum,
 and \eqref{ee} which characterizes $E_{ab}$ in the second sum.
Assuming  the convention that (see Definition \ref{defE})
$$\begin{array}{c}
E_{a_{1}a_0}=E_{a_{1}e}:=H_{a a_1^{-1}}^{\ast}B_{a}(v_{a})\in \widehat{V}_{a_1},\\ \\
 E_{a_{J+1}a_J}=E_{e a_J}:=
\overline{H_{b a_J}^{\ast}B_{b}(v_b)}\in V_{a_J}',
\end{array}$$
 we get
\begin{eqnarray}
& &<\mu[e,a,v_a],\pi(x)\mu[e,b,v_b]>=
H_{a a_1^{-1}}^{\ast} B_{a}(v_{a}, H(x^{-1})(v_{b}))\delta{(a_J b)}\nonumber\\ \nonumber\\
& &+\;\;\;
\sum_{j=1}^{J-1}E_{a_{j+1}a_j}( H(x_j)(v_{a}),H(\tilde{x}_j^{-1})(v_{b}))
\delta{(a_1 a^{-1})}\delta{(a_J b)}\nonumber\\ \nonumber\\
& &+\;\;\;
\overline{H_{b a_J}^{\ast}B_b(v_b, H_{b a_J}H(x)(v_{a}))}\delta{(a_1 a^{-1})}\nonumber\\ \nonumber\\
\;\;\;&& =\;\;\;
\sum_{j=0}^{J}E_{a_{j+1}a_j}( H(x_j)(v_{a}),H(\tilde{x}_j^{-1})(v_{b}))
\delta{(a_1 a^{-1})}\delta{(a_J b)}.\label{phi1}
\end{eqnarray}
After the position
$w_a=v_a \,\delta{(a_1 a^{-1})},$ and $ w_b=v_b \,\delta{(a_J b)},$
 each term in the sum \eqref{phi1} can be written by \eqref{xj} as
$$
E_{a_{j+1}a_j}(H(x_j)(w_{a}),H(\tilde{x}_{j}^{-1})(w_{b}))
=[\widehat{H}(\tilde{x}_j)E_{a_{j+1}a_j}H(x_j)(w_{a})](w_{b}),
$$
where at the extreme indexes
$$\begin{array}{l}
[\widehat{H}(\tilde{x}_0)E_{a_{1}a_0}H(x_0)(w_{a})](w_{b})= \widehat{H}(x)E_{a_1 e}(w_{b}),\\ \\

[\widehat{H}(\tilde{x}_J)E_{a_{J+1}a_J}H(x_J)(w_{a})](w_{b})= E_{e a_J}(H(x)w_{a}).
\end{array}$$
To conclude the proof it is sufficient to  prove the recursive formulas for the following objects
\begin{equation}\label{funoedue}
f^1_J(x)=\sum_{j=0}^{J-1}[\widehat{H}(\tilde{x}_j)E_{a_{j+1}a_j}H(x_j)(w_{a})],
\quad
 f^2_J(x)=H(x)(w_{a}).
\end{equation}

Indeed, by isolating in $f^1_J$ the term corresponding to $J-1$ we get for
$x=a_1\dots a_J,$
\begin{eqnarray*}
f^1_J(a_1\dots a_J)
&=&
\widehat{H}_{a_J a_{J-1}}f^1_{J-1}(a_1\dots a_{J-1})+
E_{a_J a_{J-1}}f^2_{J-1}(a_1\dots a_{J-1}),
\end{eqnarray*}
while
$$ f^2_J(a_1\dots a_{J})
=H_{a_J a_{J-1}}H(a_1\dots a_{J-1})(w_a)=H_{a_J a_{J-1}}f^2_{J-1}(a_1\dots a_{J-1}).$$
If $J=2$ then $x=a_1 a_{2}\neq e$. By isolating the terms for $j=0$ and $j=1,$ in
\eqref{funoedue} we obtain
$$
\begin{array}{rcccc}
f^1_2(a_1 a_{2})=&{\underbrace{\widehat{H}_{a_2 a_{1}}E_{a_1 e}}}&+&
\underbrace{E_{a_2 a_{1}}(w_{a})}, &\text{and}\quad f^2_2(a_1 a_{2})=H_{a_2 a_{1}}(w_{a}),\\
&{j=0}& &{j=1}&
\end{array}
$$
as desired. The Lemma is so completely proved.\hfill $\Box$

{\bf Proof of Theorem \ref{mumu0}}
Let $x=a_1\dots a_J\in\Gamma(c)\cap\widetilde{\Gamma}(d)$, so that $a_1=c$
and $a_J=d$. Let $R(d)$
be the row vector defined in \eqref{erre}. Also define
 the column vector
$$S^{(J-1)}=(\begin{array}{cccc}
( S^{(J-1)}_{1,1,c'})_{ c'\in A}& (S^{(J-1)}_{1,2,c'})_{ c'\in A}&
(S^{(J-1)}_{2,1,c'})_{ c'\in A}&
(S^{(J-1)}_{2,2,c'})_{ c'\in A}
\end{array})^{\top}$$
 as follows: for any $c'\in A,$ $h,i=1,2,$
\begin{equation}\label{vetts}
S^{(J-1)}_{h,i,c'}=
\left(\sum_{\atopn{x\in\Gamma(c)\cap\widetilde{\Gamma}(d')}{| x| =J-1}}
{f_{J-1}^i(x)\otimes \overline{f_{J-1}^h(x)}}\right)\delta(d'^{-1}c').
\end{equation}

We can write $|<\mu[e,a,v_a],\pi(x)\mu[e,b,v_b]>|^2$ as the product of
\eqref{mum} by its
complex conjugate. We obtain, by the properties of the tensor product,
\begin{eqnarray}
& &
\sum_{\atopn{x\in\Gamma(c)\cap\widetilde{\Gamma}(d)}{| x| =J}}
{|<\mu[e,a,v_a],\pi(x)\mu[e,b,v_b]>|^2}\nonumber\\ \nonumber\\
&=&
\sum_{\atopn{d'\in A,}{ d d'\neq e}}
\sum_{\atopn{x\in\Gamma(c)\cap\widetilde{\Gamma}(d')}{| x| =J-1}}
R(d)
({\mathcal{{\widetilde D}}}\otimes\overline{\mathcal{{\widetilde D}}})\nonumber\\
\nonumber\\
& &
\left(\begin{array}{c}
(f^1_{J-1}(x)\delta(d'^{-1}c'))_{c'\in A}\\   \\
(f^2_{J-1}(x)\delta(d'^{-1}c'))_{c'\in A}
\end{array}\right)\otimes
\left(\begin{array}{c}
(\overline{f^1_{J-1}(x)}\delta(d'^{-1}c'))_{c'\in A}\\   \\
\overline{(f^2_{J-1}(x)}\delta(d'^{-1}c'))_{c'\in A}
\end{array}\right)\nonumber\\ \nonumber \\
&=&R(d)\mathcal{D}S^{(J-1)}\label{rs},
\end{eqnarray}

By Lemma \ref{mum0} both $f^1_{J-1}$ and $f^2_{J-1}$ can be
written in terms
of $f^1_{J-2}$ and $f^2_{J-2}$ , hence  easily we obtain
$S^{(J-1)}={\mathcal{D}}S^{(J-2)},$ and
$S^{(2)}={\mathcal{D}}S^{(1)}.$

For example, if $J> 2,$
$$S^{(J-1)}_{h,i,d'}=
\sum_{\atopn{c'\in A}{d'c'\neq e}}
\sum_{\atopn{x\in\Gamma(c)\cap\widetilde{\Gamma}(c')}{| x| =J-2}}
{f^i_{J-1}(xd')\otimes \overline{f^h_{J-1}(xd')}},$$
where, by tensor products:
\begin{eqnarray*}
f^1_{J-1}(xd')\otimes\overline{f^1_{J-1}(xd')}
 &=&\widehat{H}_{d' c'}\otimes\overline{\widehat{H}}_{d' c'}
(f^1_{J-2}(x)\otimes\overline{{f^1_{J-2}(x)}})\\ \\
& &+
\widehat{H}_{d' c'}\otimes\overline{E}_{d' c'}
(f^1_{J-2}(x)\otimes\overline{{f^2_{J-2}(x)}})\\ \\
& &+E_{d' c'}\otimes\overline{\widehat{H}}_{d' c'}
(f^2_{J-2}(x)\otimes\overline{{f^1_{J-2}(x)}})\\ \\
& &+
E_{d' c'}\otimes\overline{E}_{d' c'}(f^2_{J-2}(x)\otimes\overline{{f^2_{J-2}(x)}}).
\end{eqnarray*}
So that, for $h=i=1,$ we get in particular
\begin{eqnarray*}
& \!\!\! &S^{(J-1)}_{1,1,d'}\\
&  \!\!\! \!\!\!=& \!\!\!\sum_{c'\in A}(1-\delta(d'c'))\left[
\widehat{H}_{d' c'}\otimes\overline{\widehat{H}}_{d' c'}(S^{(J-2)}_{1,1,c'})+
\widehat{H}_{d' c'}\otimes\overline{E}_{d' c'}(S^{(J-2)}_{2,1,c'})\right.\\ \\
& & +\left. E_{d' c'}\otimes\overline{\widehat{H}}_{d' c'}(S^{(J-2)}_{1,2,c'})+
E_{d' c'}\otimes\overline{E}_{d' c'}(S^{(J-2)}_{2,2,c'})\right],
\end{eqnarray*}
in other words $S^{(J-1)}_{1,1,d'}$ is the term corresponding to $d'\in A$ in the product
$$\left(\begin{array}{cccc}
(\widehat{H}_{a b}\otimes\overline{\widehat{H}}_{a b})_{a,b}&
 \!\!\!(E_{a b}\otimes\overline{\widehat{H}}_{a b})_{a,b}&
 \!\!\!(\widehat{H}_{ a b}\otimes\overline{E}_{ a b})_{a, b}&
 \!\!\!(E_{ a b}\otimes\overline{E}_{ a b})_{a, b}
\end{array}\!\!\right)S^{(J-2)}.$$
Analogously we can proceed for all other elements in the vector $S^{(J-1)}.$
If instead $J=2$ then $x=cd\neq e$ and
$$S^{(2)}_{h,i,d}=f^i(cd)\otimes\overline{f^h(cd)},$$
where, as seen in the proof of Lemma \ref{mum0}
$$
f^1(cd)={\widehat{H}_{dc}E_{c e}}+E_{d c}(w_{a}),\quad\text{and}\quad
f^2(cd)=H_{dc}(w_{a}).$$
It follows that $S^{(2)}={\mathcal D} S^{(1)},$
where, from the definition of $S^{(J)}$ in \eqref{vetts} and $S(c)$ in \eqref{esse}
is easy to see that effectively
$
S^{(1)}
=S(c).
$
So
$$S^{(J-1)}={\mathcal D} S^{(J-2)}={\mathcal D}^{J-2} S^{(1)},$$
and from \eqref{rs} we get the desired statement 
$$
\sum_{\atopn{x\in\Gamma(c)\cap\widetilde{\Gamma}(d)}{|x|=J}}
{<\mu[e,a,v_a],\pi(x)\mu[e,b,v_b]>|^2}
=R(d)\,{\mathcal D}^{J-1}\,S(c).
$$
\hfill$\Box$
\bibliographystyle{amsalpha}
\bibliography{hyp0nw}

\end{document}